\documentclass[english]{article}
\usepackage{amsmath,amssymb,amsthm,mathtools}
\usepackage[shortlabels]{enumitem}
\usepackage[pdftex,colorlinks,backref=page,citecolor=blue,bookmarks=false]{hyperref}
\usepackage{tikz}
\usetikzlibrary{shapes.misc,calc,intersections,patterns,decorations.pathreplacing}
\usetikzlibrary{arrows,shapes,positioning}
\usetikzlibrary{decorations.markings}
\usepackage{cleveref}

\allowdisplaybreaks

\usepackage{setspace}
\theoremstyle{plain}
\newtheorem{theorem}{Theorem}[section]		
\newtheorem{lemma}[theorem]{Lemma}
\newtheorem{claim}[theorem]{Claim}
\newtheorem{proposition}[theorem]{Proposition}

\theoremstyle{remark}
\newtheorem{question}[theorem]{Question}

\def\C{\mathcal{C}}
\def\X{\mathcal{X}}

\def\B{\mathcal{B}}

\def\PP{\mathcal{P}}
\def\Ex{\mathbb{E}}

\let\phi\varphi

\def \Sp {S^+}
\def \Sm {S^-}
\def \sp {s^+}
\def \smm {s^-}
\def \Ap {A^+}
\def \Am {A^-}
\def \Np {N^+}
\def \Nm {N^-}
\def \dpp {d^+}
\def \dm {d^-}
\def \Up {U^+}
\def \Um {U^-}

\def \Vgp {V_{\good}^+}
\def \Vbp {V_{\bad}^+}
\def \Vbm {V_{\bad}^-}
\def \Vgp {V_{\good}^+}
\def \Vgm {V_{\good}^-}
\def \Vo {V_{\okay}}
\def \Vg {V_{\good}}
\def \Vb {V_{\bad}}
\def \ap {a^+}
\def \am {a^-}

\def \Vgm {V_{\good}^-}

\def \Nmu{N^{\mu}_{\good}}

\def \Nnuo{N^{\nu}_{\okay}}

\def \Npo{N^+_{\okay}}
\def \Nmo{N^-_{\okay}}
\def \Npg{N^+_{\good}}
\def \Nmg{N^-_{\good}}

\def \Dm {\Delta^-}

\def \Np {N^+}
\def \Nm {N^-}

\renewcommand{\Pr}{\mathbb{P}}

\let\emptyset\varnothing

\let\originalleft\left
\let\originalright\right
\renewcommand{\left}{\mathopen{}\mathclose\bgroup\originalleft}
\renewcommand{\right}{\aftergroup\egroup\originalright}

\makeatletter
\def\imod#1{\allowbreak\mkern10mu({\operator@font mod}\,\,#1)}
\makeatother

\newcommand{\sm}{\setminus}

\newcommand{\A}{\mathcal{A}}

\DeclareMathOperator{\bad}{bad}
\DeclareMathOperator{\good}{good}
\DeclareMathOperator{\okay}{okay}

\usepackage{framed}

\title{Partitioning a tournament into sub-tournaments of high connectivity }
\author{Ant\'onio Gir\~ao\thanks{
		Mathematical Institute, 
		University of Oxford,
		Andrew Wiles Building, 
		Radcliffe Observatory Quarter, 
		Woodstock Road,
		Oxford, UK.
		Email: \texttt{girao}@\texttt{maths.ox.ac.uk}.
		Research supported by EPSRC grant EP/V007327/1.
	}
	\and
	Shoham Letzter\thanks{
		Department of Mathematics, 
		University College London, 
		Gower Street, London WC1E~6BT, UK. 
		Email: \texttt{s.letzter}@\texttt{ucl.ac.uk}. 
		Research supported by the Royal Society.   
	}
}
\date{}

\begin{document}

\maketitle

\begin{abstract}
	\setlength{\parskip}{\medskipamount}
    \setlength{\parindent}{0pt}
    \noindent

	We prove that there exists a constant $c > 0$ such that the vertices of every strongly $c \cdot kt$-connected tournament can be partitioned into $t$ parts, each of which induces a strongly $k$-connected tournament. This is clearly tight up to a constant factor, and it confirms a conjecture of K\"uhn, Osthus and Townsend (2016).

\end{abstract}

\section{Introduction}

	A classical result of Hajnal \cite{hajnal1983partition} and Thomassen \cite{thomassen1983graph} asserts that for every integer $k \ge 1$ there exists an integer $K$ such that the vertices of every $K$-connected graph can be partitioned into two sets inducing $k$-connected subgraphs. There is now a whole area of combinatorial problems concerned with questions of this type; namely, to understand whether for a certain (di)graph property any (di)graph which \textit{strongly} satisfies that property has a partition into many parts where each part still has the property. In this paper, we consider the analogue of Hajnal and Thomassen's results for tournaments.

	A digraph $D$ is said to be \emph{strongly connected} if for every $u, v \in V(D)$ there is a directed path from $u$ to $v$, and it is \emph{strongly $k$-connected} if $|D| \ge k+1$ and $D \setminus Z$ is strongly connected for every subset $Z \subseteq V(D)$ of size at most $k$. Recall that a \emph{tournament} is an orientation of a complete graph. 
	Thomassen asked (see \cite{reid1989three}) if for every sequence $k_1, \ldots, k_t$ of positive integers there exists $K$ such that if $T$ is a strongly $K$-connected tournament then there is a partition $\{V_1, \ldots, V_t\}$ of $V(T)$ such that $T[V_i]$ is $k_i$-connected for every $i \in [t]$. Denote the minimum such $K$ by $f_t(k_1, \ldots, k_t)$ (and put $f_t(k_1, \ldots, k_t) := \infty$ if there is no such $K$).

	It is easy to see that $f_t(k, 1, \ldots, 1) \le k + 3t - 3$. Chen, Gould and Li \cite{chen2001partitioning} proved that every strongly $t$-connected tournament on at least $8t$ vertices can be partitioned into $t$ strongly connected tournaments (this is clearly optimal, apart from the assumption on the number of vertices). The existence of $f_2(2,2)$ remained open until the work of K\"uhn, Osthus and Townsend \cite{kuhn2016proof} who proved that $f_t(k_1, \ldots, k_t)$ is finite for all positive integers $k_1, \ldots, k_t$. Specifically, they showed $f_t(k, \ldots, k) = O(k^7 t^4)$ and conjectured $f_t(k, \ldots, k) = O(kt)$ (which would be tight up to the implicit constant factor). More recently, Kang and Kim~\cite{KangKim} proved a better upper bound on  $f_t(k,\ldots ,k)$ showing that any tournament on $n$ vertices which is $O(k^4t)$-strongly connected can be partitioned into $t$ strongly connected tournaments where each part has a prescribed size provided all sizes are $\Omega(n)$. 
	Our main result proves the conjecture of K\"uhn, Osthus and Townsend.  
	
	\begin{theorem} \label{thm:main}
		There exists a constant $c > 0$\footnote{It probably suffices to take $c = 10^{100}$.} such that for every positive integers $k$ and $t$, if $T$ is a strongly $c \cdot kt$-connected tournament, then there is a partition $\{V_1, \ldots, V_{t}\}$ of $V(T)$ such that $T[V_i]$ is strongly $k$-connected for $i \in [t]$.
	\end{theorem} 

	We give an overview of the proof in \Cref{sec:overview}, state a few simple probabilistic tools in \Cref{sec:prelims}, and dive into the proof of \Cref{thm:main} in \Cref{sec:proof}. We conclude the paper in \Cref{sec:conclusion} with some open problems.

	Throughout the paper, when we say a tournament is \emph{$k$-connected} we mean that it is strongly $k$-connected, and by a \emph{path} we mean a directed path. We will omit floor and ceiling signs whenever it does not affect the argument.

\section{Overview of proof} \label{sec:overview}

	Let $c$ be a large constant, and let $G$ be a $c\cdot kt$-connected tournament.
	We start the proof by finding $\Omega(kt)$ pairwise disjoint `gadgets' $U(\alpha)$, with $\alpha \in \A$ for some index set $\A$, with special sets $\Sp(\alpha), \Sm(\alpha) \subseteq U(\alpha)$, such that the following properties hold:
	for every $u \in \Sm(\alpha)$ and $v \in \Sp(\alpha)$, there is a directed path in $U(\alpha)$ from $u$ to $v$; most vertices in $G$ have an out-neighbour in all but at most $kt$ sets $\Sm(\alpha)$; and similarly for in-neighbours in $\Sp(\alpha)$ (see \Cref{subsec:gadgets}). We note that similar gadgets are constructed in \cite{kuhn2016proof}. One new ingredient allows us to obtain the following additional property: there is a vertex $\sp(\alpha) \in \Sp(\alpha)$ such that almost every in-neighbour of $\sp(\alpha)$ is also an in-neighbour of $u$, for all but $O(1)$ vertices $u \in U(\alpha)$; and there exists $\smm(\alpha) \in \Sm(\alpha)$ with the analogous property for out-neighbours.

	To sketch the remainder of the proof, let us pretend that \emph{all} vertices have out-neighbours in all but at most $kt$ sets $\Sm(\alpha)$ and in-neighbours in all but at most $kt$ sets $\Sp(\alpha)$.
	We now proceed in four steps.

	In the first step (given in \Cref{subsec:available}) we remove some of the gadgets, deterministically and randomly, so that every vertex $u$ in a surviving gadget $U(\alpha)$ has $\Omega(kt)$ out- and in-neighbours that are either in $U(\alpha)$ or are not in a surviving gadget. Here it is crucial to have the latter property regarding $\sp(\alpha)$ and $\smm(\alpha)$, because effectively this means that we need to guarantee that $u$ satisfies the above property for $O(1)$ vertices $u$ in $U(\alpha)$, even if $U(\alpha)$ itself is large.

	In the second step (see \Cref{subsec:eligible}) we find $\Theta(t)$ disjoint groups of $\Theta(k)$ gadgets, such that every vertex $u$ in one of these gadgets $U(\alpha)$ has $\Omega(kt)$ out- and in-neighbours (either in $U(\alpha)$ or outside of these gadgets), each of which has an out-neighbour in $\Sm(\beta)$ for all but at most $t$ gadgets in $U(\alpha)$'s group, and an in-neighbour in $\Sp(\beta)$ for all but at most $t$ gadgets in the same group. To achieve this, we randomly partition the collection of gadgets from the previous step into $\Theta(t)$ parts, and then remove some of the parts and some of the gadgets.
	
	The third step (see \Cref{subsec:connected}) finds $t$ disjoint $k$-connected sets, each containing at least $10k$ gadgets. To do this, we first randomly assign each of the vertices not covered by the gadgets described in the previous paragraph into one of the groups of gadgets, and show that with positive probability, many of these augmented groups of gadgets contain a $k$-connected set.

	Finally, in \Cref{subsec:partition}, we assign each uncovered vertex $u$ to a $k$-connected set $U$ found in the previous paragraph which has at least $k$ in- and out-neighbours of $u$. 
	(The assumption that each group contains at least $10k$ gadgets helps here.)

	Recall, though, that this proof sketch assumed that every vertex has an out-neighbour in all but at most $kt$ sets $\Sm(\alpha)$ and similarly for in-neighbours. This need not be the case, however, and that complicates each of the above four steps. Let $\Vgp$ be the set of vertices that have out-neighbours in all but at most $kt$ sets $\Sm(\alpha)$, and define $\Vgm$ similarly. In the first step, instead of aiming for $\Omega(kt)$ out-neighbours not covered by gadgets, we aim for either $\Omega(kt)$ out-neighbours in $\Vgp \cap \Vgm$, or $\Omega(kt)$ out-neighbours in $\Vgp$, each of which has $\Omega(kt)$ in-neighbours in $\Vgp \cap \Vgm$, etc. We make similar adjustments in other steps.

\section{Notation and preliminaries} \label{sec:prelims}
	In this section we state a few probabilistic results. The following is a corollary of Hoeffding's inequality. 

	\begin{proposition} \label{prop:hoeffding}
		Let $\eta_1, \eta_2$ satisfy $\eta_1 > 4\eta_2 > 0$ and suppose that $m_1, \ldots, m_r \in [0, \eta_2 \ell]$ satisfy $m_1 + \ldots + m_r \ge \eta_1 \ell$. If $X_1, \ldots, X_r$ are independent random variables such that $X_j$ takes values $0$ and $m_j$ and $\Pr[X_j = m_j] \ge 1/2$, then
		\begin{equation} \label{eqn:hoeffding}
			\Pr[X_1 + \ldots + X_r \ge \eta_2 \ell] \ge 1 - \exp(-\eta_1/8\eta_2).
		\end{equation}
	\end{proposition}

	\begin{proof}
		Notice that we may assume that $m_1 + \ldots + m_r = \eta_1 \ell$, by possibly decreasing some values of $m_i$ and noticing that the probability in \eqref{eqn:hoeffding} does not increase by this operation.
		Write $X := X_1 + \ldots + X_r$.
		By Hoeffding's inequality,
		\begin{align*}
			\Pr\left( X \le \eta_2 \ell \right) 
			\le \Pr\left(\Ex[X] - X \ge \frac{\eta_1 \ell}{2} - \eta_2 \ell \right) 
			& \le \Pr\left(\Ex[X] - X \ge \frac{\eta_1 \ell}{4} \right) \\
			& \le \exp\left( -\frac{2(\eta_1 \ell)^2}{16\sum_{i \in [r]} m_i^2}\right) \\
			& \le \exp \left( -\frac{(\eta_1\ell)^2}{8 \cdot \frac{\eta_1 \ell}{\eta_2 \ell} \cdot (\eta_2 \ell)^2} \right)
			= \exp\left(-\frac{\eta_1}{8\eta_2}\right). \qedhere
		\end{align*}
	\end{proof}

	The next proposition is a simple probabilistic observation that we will use many times.

	\begin{proposition} \label{prop:markov}
		Let $X_1, \ldots, X_r$ be $0,1$-random variables. Suppose that $\Pr[X_i = 1] \ge 1 - \eta^2$ for every $i \in [r]$. Then 
		\begin{equation*}
			\Pr[X_1 + \ldots + X_r \ge (1 - \eta)r] \ge 1 - \eta.
		\end{equation*}
	\end{proposition}

	\begin{proof}
		Write $X := X_1 + \ldots + X_r$.
		By assumption, $\Ex[X] \ge (1 - \eta^2)r$ and $X \le r$. Write $p := \Pr[X \ge (1 - \eta)r]$. Then $\Ex[X] \le p \cdot r + (1 - p) \cdot (1 - \eta)r = (1 - \eta(1 - p)) r$. It follows that $(1 - \eta^2)r \le (1 - \eta(1 - p))r$, implying $\eta(1 - p) \le \eta^2$, i.e.\ $p \ge 1 - \eta$, as claimed.
	\end{proof}

	To conclude the section, we state Chernoff's bounds, which we will use extensively.

	\begin{lemma} \label{lem:chernoff}
		Let $X$ be the sum of independent random variables taking values in $\{0, 1\}$, and write $\mu := \Ex[X]$. Then the following holds for $\delta \in [0, 1]$.
		\begin{align*}
			& \Pr[X \le (1-\delta)\mu] \le \exp(-\delta^2 \mu / 2) \\
			& \Pr[X \ge (1+\delta)\mu] \le \exp(-\delta^2 \mu / 3).
		\end{align*}
	\end{lemma}

\section{The proof} \label{sec:proof}

	In this section we prove our main theorem, \Cref{thm:main}.

	\begin{proof}[Proof of \Cref{thm:main}]
		Observe that it suffices to prove the existence of a suitable constant $c$ such that for \emph{large enough} $k$ and every $t \ge 2$, the vertices of every $c \cdot kt$-connected tournament can be partitioned into $t$ sets which induce $k$-connected tournaments. Throughout our proof we indeed assume that $k$ is large enough.

		Pick constants $\rho, \sigma_1, \sigma_2, \sigma_3, \tau_1, \tau_2, \tau_3$ as follows.
		\begin{align} \label{eqn:values}
			\begin{split}
				& \rho = 10^4 \qquad \sigma_1 = 10^{60} \qquad \sigma_2 = 10^4 \qquad \sigma_3 = 10. \\
				& k \gg \tau_1 \gg \tau_2 \gg \tau_3 \gg \rho, \sigma_1, \sigma_2, \sigma_3.
			\end{split}
		\end{align}
		Let $T$ be a $\tau_1 kt$-connected tournament with vertex set $V$.
		Our aim is to find a partition $\{V_1, \ldots, V_t\}$ of $V$ such that $T[V_i]$ is $k$-connected for every $i \in [t]$. This will be done in five steps.

	\subsection{Building gadgets} \label{subsec:gadgets}
		The first step in our proof is the construction of $\sigma_1 kt$ gadgets, which are the sets $U(\alpha)$ obtained by the following proposition.
		The construction of the gadgets is done similarly to \cite{kuhn2016proof} (see page 6), with several differences. First, the size of the sets $\Sm(\alpha)$ and $\Sp(\alpha)$ (corresponding to $A_i$ and $B_i$ in \cite{kuhn2016proof}) is much smaller than in \cite{kuhn2016proof}. Second, we construct $\sigma_1 kt$ gadgets, which is quite a lot more gadgets than we need ($10kt$) for the partition, to allow for some flexibility (in later steps we will discard some of the gadgets, randomly and deterministically). Third, a minimality assumption on the paths $P(\alpha)$ which join the sets $\Sm(\alpha)$ and $\Sp(\alpha)$ (see the three paragraphs before \Cref{fig:gadget}) allows us to find a small set $X(\alpha)$ as in \ref{itm:gadget-neighbourhood}. Finally, to compensate for the smaller size of $\Sm(\alpha)$ and $\Sp(\alpha)$ we need to consider sets $\Vbm$ and $\Vbp$, which are relatively small sets of exceptional vertices. The third point is probably the most crucial new ingredient here.

		\begin{proposition} \label{prop:gadgets}
			There exist sets of vertices $\Sp(\alpha), \Sm(\alpha) \subseteq S(\alpha) \subseteq U(\alpha) \subseteq V$ and $X(\alpha) \subseteq V$, indexed by a set $\A_1$ of size $\sigma_1 kt$, and vertices $\sp(\alpha), \smm(\alpha) \in S(\alpha)$, satisfying the following properties.
			\begin{enumerate} [label = \rm(G\arabic*)]
				\item \label{itm:gadget-disjoint}
					The sets $U(\alpha)$, with $\alpha \in \A_1$, are pairwise disjoint.
				\item \label{itm:gadget-Y-small}
					$|S(\alpha)| \le \rho$ and $|X(\alpha)| \le \rho \sigma_1 kt$, for $\alpha \in \A_1$.
				\item \label{itm:gadget-path}
					For every $\alpha \in \A_1$, $u \in \Sm(\alpha)$ and $v \in \Sp(\alpha)$, there is a path in $T[U(\alpha)]$ from $u$ to $v$.
				\item \label{itm:gadget-neighbourhood}
					Every in-neighbour of $\sp(\alpha)$ in $U(\alpha) \sm X(\alpha)$ is also an in-neighbour of every vertex in $U(\alpha) \setminus S(\alpha)$, for $\alpha \in \A_1$.
					Analogously for $\smm(\alpha)$ with respect to out-neighbours.
			\end{enumerate}

			Let $\Vo = V \setminus \bigcup_{\alpha} S(\alpha)$, let $\Vbp$ be the set of vertices in $\Vo$ that have no out-neighbours in at least $kt$ sets $\Sm(\alpha)$, and let $\Vbm$ be the set of vertices in $\Vo$ that have no in-neighbours in at least $kt$ sets $\Sp(\alpha)$.
			\begin{enumerate} [resume, label = \rm(G\arabic*)]
				\item \label{itm:gadget-good}
					Every $u \in \Vo$ satisfies $\dpp(u) \ge 10^{12} \cdot |\Vbp|$ and $\dm(u) \ge 10^{12} \cdot |\Vbm|$.
			\end{enumerate}
		\end{proposition}

		\begin{proof}

			Let $\Ap$ be the set of $\sigma_1 kt$ vertices in $V$ with largest out-degrees (breaking ties arbitrarily), and let $\Am$ be the set of $\sigma_1 kt$ vertices in $V$ with largest in-degrees (note that $|T| > 2 \sigma_1 kt$, so we may assume that $\Ap$ and $\Am$ are disjoint). We define sets $\Sp(a)$, for $a \in \Ap$, and sets $\Sm(a)$, for $a \in \Am$, as follows.

			Let $a_1, \ldots, a_{|\Ap|}$ be an arbitrary ordering of the vertices in $\Ap$. Having defined $\Sp(a_1), \ldots, \Sp(a_{i-1})$, let $\Up_i := V \setminus (\Ap \cup \Am \cup \Sp(a_1) \cup \ldots \cup \Sp(a_{i-1}))$, so $\Up_i$ is the set of vertices that are currently unused. 
			Pick a sequence $u_{i,0}, \ldots, u_{i,m_i}$ as follows. Set $u_{i,0} := a_i$. Having defined $u_{i,0}, \ldots, u_{i,j-1}$, let $\Up_{i,j}$ be the set of vertices in $\Up_i \setminus \{u_{i,1}, \ldots, u_{i,j-1}\}$ that do not have an in-neighbour in $\{u_{i,0}, \ldots, u_{i,j-1}\}$, and take $u_{i,j}$ to be the vertex of maximum out-degree in $T[\Up_{i,j}]$; if $\Up_{i,j}$ is empty or if $j > \rho / 10$, set $m_i := j-1$ and define $\Sp(a_i) := \{u_{i, 0}, \ldots, u_{i, m_i}\}$.
			We have thus defined sets $\Sp(a)$ for $a \in \Ap$.

			Note that $|\Up_{i,1}| \le \dm(a_i)$ and $u_{i,j}$ has out-degree at least $(|\Up_{i, j}| - 1)/2$ in $U_{i, j}$, for $i \in [|\Ap|]$ and $j \in [m_i]$ (using that $T$ is a tournament). It follows that $|\Up_{i, j}| \le 2^{-(j-1)} \cdot \dm(a_i)$, implying that the number of vertices in $\Up_{i+1}$ that have no in-neighbours in $\Sp(a_i)$ is at most 
			\begin{equation} \label{eqn:bad-degree}
				\frac{\dm(a_i)}{2^{\rho/10}} 
				\le \max_{a \in \Ap} \frac{\dm(a)}{2^{\rho/10}} 
				\le \max_{a \in \Ap} \frac{\dm(a)}{10^{12} \cdot \sigma_1},
			\end{equation}
			using that $\rho = 10^4$, $\sigma_1 = 10^{60}$ and $2^{10} \ge 10^3$ (see \eqref{eqn:values}).
			Also observe that $\Sp(a)$ is a set of size at most $\rho / 10$ that induces a transitive tournament whose sink is $a$, for $a \in \Ap$. 

			We now pick sets $\Sm(a)$, with $a \in \Am$, similarly. 
			Let $a_1, \ldots, a_{|\Am|}$ be an ordering of the vertices in $\Am$. Having defined $\Sm(a_1), \ldots, \Sm(a_{i-1})$, define $\Um_i$ to be the set of unused vertices, namely
			\begin{equation*}
				\Um_i := V \setminus \left( \Big(\bigcup_{a \in \Ap} \Sp(a)\Big) \cup \Ap \cup \Am \cup \Sm(a_1) \cup \ldots \cup \Sm(a_{i-1})\right).
			\end{equation*}
			Let $v_{i,0}, \ldots, v_{i, m_i}$ be chosen as follows. Take $v_{i,0} := a_i$. Having defined $v_{i,0}, \ldots, v_{i, j-1}$, let $\Um_{i,j}$ be the set of vertices in $\Um_i \setminus \{v_{i,1}, \ldots, v_{i,j-1}\}$ which do not have an out-neighbour in $\{v_{i,0}, \ldots, v_{i,j-1}\}$. Take $v_{i, j}$ to be a vertex of maximum in-degree in $T[\Um_{i,j}]$. If $\Um_{i,j}$ is empty or $j > \rho/10$, define $m_i := j-1$ and put $\Sm(a_i) := \{v_{i,0}, \ldots, v_{i,m_i}\}$. 
			As above, the number of vertices in $\Um_{i+1}$ that have no out-neighbour in $\Sm(a_i)$ is at most the maximum of $\dpp(a) / 10^{12} \sigma_1$ over $a \in \Am$, and $\Sm(a)$ is a set of size at most $\rho / 10$ that induces a transitive tournament whose source is $a$, for $a \in \Am$. 

			Denote $S := \bigcup_{a \in \Ap} \Sp(a) \cup \bigcup_{a \in \Am} \Sm(a)$; then $|S| \le \rho \sigma_1 kt$.
			Because $T$ is $\tau_1 kt$-connected and $\tau_1 \ge \rho \sigma_1$ (see \eqref{eqn:values}), there is a collection $\PP$ of $\sigma_1 kt$ pairwise vertex-disjoint paths, each of which starts at the sink of $\Sm(a)$ for some $a \in \Am$ and ends at the source of $\Sp(a')$ for some $a' \in \Ap$, and which do not contain any vertices of $S \setminus (\Ap \cup \Am)$. We will assume that $\PP$ is \emph{minimal}, meaning that for every collection $\PP'$ of paths with the above properties, the number of vertices covered by paths in $\PP'$ is at least as large as the number of vertices covered by paths in $\PP$.

			To denote these paths and the corresponding pairing of sets $\Sp(a)$ with sets $\Sm(a')$, let $\A_1$ be a set of size $\sigma_1 kt$, which will serve as the set of indices, and for each $\alpha \in \A_1$, let $P(\alpha)$ be one of the paths above, let $\smm(\alpha)$ be the start vertex of $P(\alpha)$ and let $\sp(\alpha)$ be the last vertex in $P(\alpha)$. Let $\ap(\alpha) \in \Ap$ and $\am(\alpha) \in \Am$ be such that $\smm(\alpha) \in \Sm(\am(\alpha))$ and $\sp(\alpha) \in \Sp(\ap(\alpha))$.
			We abuse notation slightly by denoting $\Sp(\alpha) := \Sp(\ap(\alpha))$ and $\Sm(\alpha) := \Sm(\am(\alpha))$. 

			Define $U(\alpha) := \Sp(\alpha) \cup \Sm(\alpha) \cup V(P(\alpha))$, and take $S(\alpha)$ to be the union of $\Sp(\alpha) \cup \Sm(\alpha)$ with the first three and last three vertices in $P(\alpha)$, or with the whole of $V(P(\alpha))$ if it has at most five vertices (see \Cref{fig:gadget} for an illustration of these sets and vertices $\sp(\alpha), \smm(\alpha)$).
			For a subset $\A \subseteq \A_1$, write $U(\A) := \bigcup_{\alpha \in \A} U(\alpha)$. 

			\begin{figure}[ht]
				\begin{centering}
					\includegraphics[scale = 1]{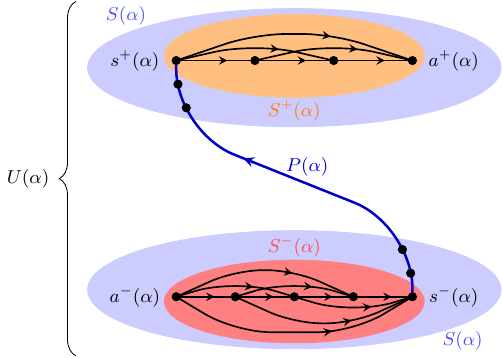}
					\caption{The sets $U(\alpha), \Sp(\alpha), \Sm(\alpha)$, vertices $\sp(\alpha), \smm(\alpha)$ and path $P(\alpha)$}
					\label{fig:gadget}
				\end{centering}
			\end{figure}

			Take $X(\alpha)$ to be the set of vertices $u$ such that $u$ is an in-neighbour of $\sp(\alpha)$ but an out-neighbour of some vertex in $U(\alpha) \sm S(\alpha)$ or $u$ is an out-neighbour of $\smm(\alpha)$ but an in-neighbour of some vertex in $U(\alpha) \sm S(\alpha)$.

			\begin{claim} \label{claim:nbds-paths}
				Let $\alpha \in \A_1$. Then all but at most $\rho\sigma_1 kt/2$ out-neighbours of $\smm(\alpha)$ are out-neighbours of all vertices in $U(\alpha) \setminus S(\alpha)$. Similarly, all but at most $\rho \sigma_1 kt/2$ in-neighbours of $\sp(\alpha)$ are in-neighbours of all vertices in $U(\alpha) \setminus S(\alpha)$.   
				In particular, $|X(\alpha)| \le \rho \sigma_1 kt$.
			\end{claim}
			
			\begin{proof}
				Write $s = \smm(\alpha)$, $P = P(\alpha)$ and let $u \in U(\alpha) \setminus S(\alpha)$. Then $u$ is a vertex in $P(\alpha)$ which is not one of the first three vertices.

				First note that all edges both of whose ends are in $V(P)$, and which are not edges of $P$, are directed `backwards', namely if $P = (x_1 \ldots x_t)$ and if $1 \le i < j - 1 < t$, then $x_j x_i$ is a directed edge in $T$. This is due to the minimality assumption on $\PP$; if instead $x_i x_j$ is an edge, then $P$ can be replaced by $(x_1 \ldots x_i x_j \ldots x_t)$, contradicting minimality (see the leftmost part of \Cref{fig:P}). In particular, $s$ has no out-neighbours in $V(P)$ other than the second vertex in $P$. 

				\begin{figure}[ht]
					\centering
					\includegraphics[scale = 1]{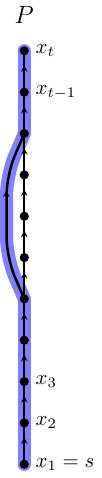}
					\hspace{2.2cm}
					\includegraphics[scale = 1]{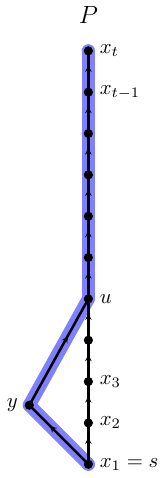}
					\hspace{2.2cm}
					\includegraphics[scale = 1]{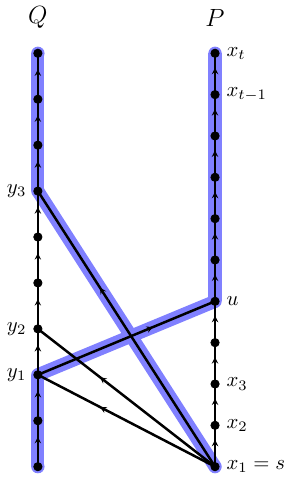}
					\caption{Contradiction to minimality in the proof of \Cref{claim:nbds-paths}}
					\label{fig:P}
				\end{figure}

				It is easy to see that every out-neighbour of $s$ which is not in $U(\A_1)$ is also an out-neighbour of $u$. Indeed, suppose to the contrary that $y$ is an out-neighbour of $s$ which is an in-neighbour of $u$. Then $P$ can be replaced by $syu P_{u \to}$ (where $P_{u \to}$ is the subpath of $P$ that starts at $u$ and follows $P$ to the end), contradicting the minimality of $\PP$ (see the middle part of \Cref{fig:P}).

				Finally, consider $\beta \in \A_1 \setminus \{\alpha\}$, and write $Q = P(\beta)$. We claim that all but the last two out-neighbours of $s$ in $Q$ are out-neighbours of $u$. Indeed, otherwise there are three out-neighbours $y_1, y_2, y_3$ of $s$ in $Q$, that appear in $Q$ in this order, such that $y_1$ is an in-neighbour of $u$. 
				Replace the paths $P$ and $Q$ by the following two path: $Q_{\to y_1} y_1 u P_{u \to}$ and $s y_3 Q_{y_3 \to}$ (where $Q_{\to y}$ is the subpath of $Q$ that starts as in $Q$ and ends at $y$, etc.).
				The vertices of these new paths are in $V(P) \cup V(Q)$, but avoid $y_2$, contradicting the minimality of $\PP$ (see the rightmost figure in \Cref{fig:P}).

				To summarise, the number of out-neighbours of $s$ that are not out-neighbours of some vertex in $U(\alpha) \setminus S(\alpha)$ is at most
				\begin{equation*}
					1 + 2(|\A_1| - 1) + \Big|\bigcup_{\beta \in \A_1} \Sp(\beta)\Big| + \Big|\bigcup_{\beta \in \A_1} \Sm(\beta))\Big| \le (2 + 2\rho / 10)\sigma_1 kt
					\le \rho \sigma_1 kt/2,
				\end{equation*}
				as claimed. An analogous argument can be used to prove the second part of the observation. 
			\end{proof}

			To complete the proof it now suffices to verify that properties \ref{itm:gadget-disjoint} to \ref{itm:gadget-good} hold. Item \ref{itm:gadget-disjoint} follows directly from the choice of sets $U(\alpha)$. It is easy to see that the first part of \ref{itm:gadget-Y-small} holds; indeed, $|S(\alpha)| \le |\Sp(\alpha)| + |\Sm(\alpha)| + 6 \le 2\rho / 10 + 6 \le \rho$; the second part follows from the last observation.
			To see that \ref{itm:gadget-path} holds, let $u \in \Sm(\alpha)$ and $v \in \Sp(\alpha)$. Then $u \smm(\alpha) P(\alpha) \sp(\alpha) v$ is a path from $u$ to $v$ in $T[U(\alpha)]$. 
			Next, notice that \ref{itm:gadget-neighbourhood} holds by definition of $X(\alpha)$.

			It remains to prove \ref{itm:gadget-good}. We form an auxiliary bipartite graph $H$, with parts $\A_1$ and $W := V \setminus \bigcup_{\alpha} S(\alpha)$, where $\alpha w$ (with $\alpha \in \A_1$ and $w \in W$) is an edge if $w$ has no in-neighbours in $\Sp(\alpha)$. Write $\Dm := \max_{a \in \Ap} \dm(a)$.
			By \eqref{eqn:bad-degree}, $d_H(\alpha) \le \frac{\Dm}{10^{12}\sigma_1}$.
			It follows that $e(H) \le |\A_1| \cdot \frac{\Dm}{10^{12}\sigma_1} \le \frac{kt\Dm}{10^{12}}$. Recall that a vertex $w$ is in $\Vbm$ if $w \in W$ and $d_H(w) \ge kt$, implying that $|\Vbm| \le e(H) / kt \le \frac{\Dm}{10^{12}}$. As every vertex $w$ in $W$ satisfies $\dm(w) \ge \Dm$ (by choice of $\Dm$ and $\Ap$), we have $\dm(w) \ge 10^{12}|\Vbm|$. A similar argument shows that $\dpp(w) \ge 10^{12}|\Vbp|$, establishing \ref{itm:gadget-good}.
		\end{proof}

	Let $S(\alpha), \Sp(\alpha), \Sm(\alpha), \sp(\alpha), \smm(\alpha), U(\alpha)$, with $\alpha \in \A_1$, be as in \Cref{prop:gadgets}. For a subset $\A \subseteq \A_1$, define
	\begin{equation*}
		S(\A) := \bigcup_{\alpha \in \A} S(\alpha), \qquad U(\A) := \bigcup_{\alpha \in \A} U(\alpha), \qquad W(\A) := V \setminus U(\alpha).
	\end{equation*}
	Let $\Vo$, $\Vbp$ and $\Vbm$ be defined as in \Cref{prop:gadgets}, namely,
	\begin{align*}
		& \Vo = V \setminus \bigcup_{\alpha \in \A_1} S(\alpha) \\
		& \Vbp = \{u \in \Vo : \text{$u$ has no out-neighbours in $\Sm(\alpha)$ for at least $kt$ indices $\alpha \in \A_1$}\} \\
		& \Vbm = \{u \in \Vo : \text{$u$ has no in-neighbours in $\Sp(\alpha)$ for at least $kt$ indices $\alpha \in \A_1$}\}.
	\end{align*}
	Define also 
	\begin{equation*}
		\Vb := \Vbp \cap \Vbm, \qquad \Vgp := \Vo \setminus \Vbp, \qquad \Vgm := \Vo \setminus \Vbm, \qquad \Vg := \Vgp \cap \Vgm.
	\end{equation*}

	Without loss of generality, we assume that $|\Vbm| \ge |\Vbp|$. The following claim establishes additional properties of the sets defined above. (Note the difference between \ref{itm:gadget-bad-out} and \ref{itm:gadget-bad-in}, which is due to the assumption $|\Vbm| \ge |\Vbp|$.)

	\begin{claim} \label{claim:additional-gadget-properties}
		The following properties holds.
		\begin{enumerate}[label = \rm(G\arabic*)]
			\setcounter{enumi}{5}
			\item \label{itm:gadget-size-V-good}
				$|\Vg| \ge n/2$.
			\item \label{itm:gadget-bad-out}
				Every vertex in $\Vo$ has at least $\max\{10^{11} |\Vbp|,\, \tau_1 kt/2\}$ out-neighbours in $\Vgp$.
			\item \label{itm:gadget-bad-in}
				Every vertex in $\Vo$ has at least $\max\{10^{11} |\Vbm|,\, \tau_1 kt/2\}$ in-neighbours in $\Vg$.
			\item \label{itm:gadget-not-okay}
				Every vertex in $V$ has at least $\max\{10^{11} |V \setminus \Vo|,\, \tau_1kt /2\}$ out- and in-neighbours in $\Vo$.
		\end{enumerate}
	\end{claim}

	\begin{proof}
		To see \ref{itm:gadget-size-V-good}, note that \ref{itm:gadget-neighbourhood} implies that $|\Vbp|, |\Vbm| \le n/10^{12}$. Additionally, $|V \setminus \Vo| \le \rho \sigma_1 kt \le \tau_1 kt / 10 \le n/10$ (using \ref{itm:gadget-Y-small}, $\tau_1 \gg \rho, \sigma_1$ and that $T$ has minimum out- and in-degree at least $\tau_1 kt$.) Altogether we have $|V \setminus \Vg| = |\Vb| + |V \setminus \Vo| \le |\Vbp| + |\Vbm| + |V \setminus \Vo| \le n/2$, with room to spare. It follows that $|\Vg| \ge n/2$, as claimed.

		Note that by \ref{itm:gadget-good} and because $T$ has minimum out-degree at least $\tau_1 kt$, every vertex in $\Vo$ has at least $\max\{10^{12} |\Vbp|, \tau_1 kt\}$ out-neighbours.
		Note also that $|V \setminus \Vgp| \le |S(\A_1)| + |\Vbp| \le \rho \sigma_1 kt + |\Vbp| \le \frac{1}{2}\max\{10^{12}|\Vbp|, \tau_1 kt\}$ (using $\tau_1 \gg \rho, \sigma_1$). Property \ref{itm:gadget-bad-out} follows.

		A similar argument implies \ref{itm:gadget-bad-in}. Indeed, by \ref{itm:gadget-good} and the minimum degree assumption on $T$, every vertex in $\Vo$ has in-degree at least $\max\{10^{12}|\Vbm|, \tau_1 kt\}$. Thus $|V \setminus \Vg| \le |S(\A_1)| + |\Vbp| + |\Vbm| \le \rho \sigma_1 kt + 2|\Vbm| \le \frac{1}{2} \max\{10^{12}|\Vbm|, \tau_1 kt\}$, using $|\Vbm| \ge |\Vbp|$, and \ref{itm:gadget-bad-in} follows.

		Finally, if $u \in V$ and $\nu \in \{+, -\}$ then $|N^{\nu}(u) \cap \Vo| \ge \tau_1 kt - \sigma_1 \rho kt \ge \max\{10^{11}|V \setminus \Vo|, \, \tau_1 kt/2\}$ (using $V \setminus \Vo = S(\A_1)$), proving \ref{itm:gadget-not-okay}.
	\end{proof}

	\subsection{Many available neighbours} \label{subsec:available}

		Our aim in the next few steps is to form $t$ pairwise disjoint sets, each consisting of $10k$ gadgets $U(\alpha)$ as well as some additional vertices, and each inducing a $k$-connected set. The simplest way to form such a set is to let $U$ be a union of $10k$ gadgets, let $W$ be a set of vertices which have out- and in-neighbours in all but at most $k$ gadgets in $U$, such that each vertex in $U$ has at least $k$ out- and in-neighbours in $W$. In practice, this requirement on $W$ is a bit too strong, e.g.\ because of the existence of vertices that have out-neighbours in few of the sets $\Sm(\alpha)$. In this subsection we trim the collection of gadgets so that each vertex in each remaining gadget $U(\alpha)$ has many out- and in-neighbours that are good candidates for being in such 
		a set $W$ for a set $U$ as above that contains $U(\alpha)$. Below, we give a formal definition for this notion and then state \Cref{prop:available} which formalises this assertion.

		Given a subset $\A \subseteq \A_1$, an element $\alpha \in \A$ and a vertex $u$, we say that $u$ is \emph{available} for $\alpha$ with respect to $\A$ if one of the following holds, where $W := W(\A \setminus \{\alpha\})$; namely $W = \big(V \setminus \bigcup_{\beta \in \A} U(\beta)\big) \cup U(\alpha)$.
		
		\begin{enumerate}[label = \rm(A\arabic*)]
			\item \label{itm:available-good}
				$u \in \Vg \cap W$.
			\item \label{itm:available-out-good}
				$u \in (\Vgp \setminus \Vgm) \cap W$ and $u$ has at least $\tau_2 kt$ in-neighbours as in \ref{itm:available-good}.
			\item \label{itm:available-in-good}
				$u \in (\Vgm \setminus \Vgp) \cap W$ and $u$ has at least $\tau_2 kt$ out-neighbours as in \ref{itm:available-good} or \ref{itm:available-out-good}.
			\item \label{itm:available-bad}
				$u \in \Vb \cap W$ and $u$ has at least $\tau_2 kt$ out-neighbours as in \ref{itm:available-good} or \ref{itm:available-out-good}, and at least $\tau_2 kt$ in-neighbours as in \ref{itm:available-good}.
		\end{enumerate}

		Our goal in this subsection is to prove the following proposition.

		\begin{proposition} \label{prop:available}
			There is a subset $\A_2 \subseteq \A_1$ of size $\sigma_2 kt$ such that $s$ has at least $\tau_2 kt$ out- and in-neighbours that are available for $\alpha$ with respect to $\A_2$, for every $\alpha \in \A_2$ and $s \in S(\alpha)$.
		\end{proposition}

		We break down the proof of \Cref{prop:available} into three lemmas. Before stating them, we need some notation.
		Define $\phi_0, \phi_1, \phi_2, \phi_3$ as follows.
		\begin{equation*}
			\phi_0 := \frac{\tau_1}{4}, \qquad
			\phi_1 := \frac{\phi_0}{2^{6}\rho}, \qquad
			\phi_2 := \frac{\phi_1}{2^{14}\rho^2}, \qquad
			\phi_3 := \frac{\phi_2}{2^{14}\rho^2}. 
		\end{equation*}
		Note that, as $\tau_1 \gg \tau_2$, we have $\phi_3 \ge \tau_2$.
		For a vertex $u$, write 
		\begin{align*}
			\begin{array}{ll}
				\Npo(u) := \Np(u) \cap \Vo \qquad \qquad
				&\Nmo(u) := \Nm(u) \cap \Vo \vspace{3pt}\\
				\Npg(u) := \Np(u) \cap \Vgp \qquad \qquad
				&\Nmg(u) := \Nm(u) \cap \Vg
			\end{array}
		\end{align*}
		(Note the difference between $\Npg(u)$ and $\Nmg(u)$).
		Then $|\Nnuo(u)| \ge \tau_1 kt / 2$ for every $u \in V$ and $\nu \in \{+, -\}$, and $|\Nmu(u)| \ge \tau_1 kt / 2$ for every $u \in \Vo$ and $\mu \in \{+, -\}$, 
		by \ref{itm:gadget-bad-out}, \ref{itm:gadget-bad-in} and \ref{itm:gadget-not-okay}.

		\begin{lemma} \label{lem:available-step-1}
			There is a subset $\C_1 \subseteq \A_1$ of size at least $|\A_1|/36\rho^2$ such that $|\Nnuo(s) \cap W(\C_1 \setminus \{\alpha\})| \ge \phi_1 kt$ for every $\alpha \in \C_1$, $s \in S(\alpha)$ and $\nu \in \{+, -\}$.
		\end{lemma}

		\begin{lemma} \label{lem:available-step-2}
			Let $\C_1$ be as in \Cref{lem:available-step-1}. Then there is a subset $\C_2 \subseteq \C_1$ of size at least $|\C_1| / 6^4\rho^4$ such that, for every $\alpha \in \C_2$, $s \in S(\alpha)$ and $\nu \in \{+, -\}$, the set $\Nnuo(s) \cap W(\C_2 \setminus \{\alpha\})$ contains at least $\phi_2 kt$ vertices $u$ satisfying $|\Nmu(u) \cap W(\C_2 \setminus \{\alpha\})| \ge \phi_2 kt$ for $\mu \in \{+, -\}$.
		\end{lemma}

		\begin{lemma} \label{lem:available-step-3}
			Let $\C_2$ be as in \Cref{lem:available-step-2}.
			Then there is a subset $\C_3 \subseteq \C_2$ of size at least $|\C_2|/6^4\rho^4$ such that for every $\alpha \in \C_3$, $s \in S(\alpha)$ and $\nu \in \{+, -\}$ the following holds: $\Nnuo(s) \cap W(\C_3 \setminus \{\alpha\})$ contains at least $\phi_3 kt$ vertices $u$ for which $\Nmu(u) \cap W(\C_3 \setminus \{\alpha\})$ contains at least $\phi_3 kt$ vertices $v$ such that $|\Nmg(v) \cap W(\C_3 \setminus \{\alpha\})| \ge \phi_3 kt$, for $\mu \in \{+, -\}$.
		\end{lemma}

		\begin{proof}[Proof of \Cref{prop:available} using \Cref{lem:available-step-1,lem:available-step-2,lem:available-step-3}]
			Let $\C_1$, $\C_2$ and $\C_3$ be as in \Cref{lem:available-step-1}, \Cref{lem:available-step-2} and \Cref{lem:available-step-3}, respectively. Define $\A_2 := \C_3$.
			Then 
			\begin{equation*}
				|\A_2| 
				= |\C_3| 
				\ge \frac{|\A_1|}{36\rho^2 \cdot 6^4\rho^4 \cdot 6^4\rho^4} 
				\ge \frac{\sigma_1 kt}{6^{10} \rho^{10}}
				\ge \sigma_2 kt.
			\end{equation*}
			(using $\rho = \sigma_2 = 10^4$ and $\sigma_1 = 10^{60}$.)

			Fix $\alpha \in \A_2$ and $s \in S(\alpha)$ and write $W := W(\A_2 \setminus \{\alpha\})$.
			We claim that every vertex $u \in W$, such that $\Nmu(u) \cap W$ contains at least $\tau_2 kt$ vertices $v$ with $|\Nmg(v) \cap W| \ge \tau_2 kt$, for $\mu \in \{+, -\}$, is available for $\alpha$ with respect to $\A_2$. To see this, we need to show that one of \ref{itm:available-good} to \ref{itm:available-bad} holds for $u$. 
			\begin{itemize}
				\item
					If $u \in \Vg$ then \ref{itm:available-good} automatically holds. 
				\item
					If $u \in \Vgp \setminus \Vgm$ then, using $|\Nmg(u) \cap W| \ge \tau_2 kt$, we see that \ref{itm:available-out-good} holds. 
				\item
					Suppose that $u \in \Vgm \setminus \Vgp$. 
					Note that every vertex $v \in \Npg(u) \cap W$ with $|\Nmg(v) \cap W| \ge \tau_2 kt$ satisfies one of \ref{itm:available-good} or \ref{itm:available-out-good}. Then, since $u$ has at least $\tau_2 kt$ such out-neighbours $v$, it satisfies \ref{itm:available-in-good}.
				\item
					Suppose now $u \in \Vb$.
					The reasoning from the previous item shows that $u$ has at least $\tau_2 kt$ out-neighbours satisfying \ref{itm:available-good} or \ref{itm:available-out-good}, and the reasoning from the second item show that $u$ has at least $\tau_2 kt$ in-neighbours satisfying \ref{itm:available-good}. In particular, \ref{itm:available-bad} holds.
			\end{itemize}
			Recalling that $\phi_3 \ge \tau_2$, it follows from the choice of $\C_3$ that $s$ has at least $\tau_2 kt$ out- and in-neighbours that are available for $\alpha$ with respect to $\A_2$. We conclude that $\A_2$ satisfies the requirements of \Cref{prop:available}.
		\end{proof}

		We now prove \Cref{lem:available-step-1,lem:available-step-2,lem:available-step-3}.

		\subsubsection{Proof of \Cref{lem:available-step-1}}
			\begin{proof}

				In order to find an appropriate set $\C_1$, we will find subsets $\C_1''', \C_1'' \subseteq \C_1' \subseteq \C_0$ and use them to define $\C_1$. The set $\C_1'$ will be taken to satisfy the properties of the following claim.

				\begin{claim} \label{claim:available-one}
					There is a subset $\C_1' \subseteq \C_0$ of size at least $|\C_0|/9\rho^2$ such that for every $\alpha \in \C_1'$, if for some $s \in S(\alpha)$ and $\nu \in \{+, -\}$ there exists $\beta \in \C_1' \setminus \{\alpha\}$ with $|\Nnuo(s) \cap U(\beta)| \ge \phi_1 kt$, then there exists $\gamma \in \C_0 \setminus \C_1'$ with $|\Nnuo(s) \cap U(\gamma)| \ge \phi_1 kt$.
				\end{claim}

				\begin{proof}
					Fix an arbitrary ordering $\prec$ of $\C_0$. 
					First, run the following process. Start with $X_1 = Y_1 = \emptyset$. As long as the set $\C_0 \setminus (X_1 \cup Y_1)$ is non-empty, take $\alpha$ to be the first element this set (according to $\prec$) and put it in $X_1$. Then, for each $s \in S(\alpha)$ and $\nu \in \{+, -\}$, if $|\Nnuo(s) \cap U(\beta)| \ge \phi_1 kt$ for some $\beta \in \C_0 \setminus (X_1 \cup Y_1 \cup \{\alpha\})$, put one such $\beta$ in $Y_1$.

					Next, we run a similar process on $X_1$. Start with $X_2 = Y_2 = \emptyset$. As long as the set $X_1 \setminus (X_2 \cup Y_2)$ is non-empty, let $\alpha$ be the \emph{last} element in this set (according to $\prec$) and put $\alpha$ in $X_2$. Then, for every $s \in S(\alpha)$ and $\nu \in \{+, -\}$, if $|\Nnuo(s) \cap U(\beta)| \ge \phi_1 kt$ for some $\beta \in X_1 \setminus (X_2 \cup Y_2 \cup \{\alpha\})$, put one such $\beta$ in $Y_2$.

					Now set $\C_1' := X_2$. We claim that this choice satisfies the requirements of the claim.
					To see this observe first that $|X_1| \ge |\C_0|/(2\max_{\alpha} |S(\alpha)| + 1) \ge |\C_0|/3\rho$, because $\{X_1, Y_1\}$ is a partition of $\C_0$ and for each element $\alpha$ which is put in $X_1$, at most $2|S(\alpha)|$ elements are put in $Y_1$. Similarly, $|X_2| \ge |X_1| / 3\rho \ge |\C_0|/9\rho^2$. 
					Next, suppose that for some $\alpha \in X_2$ there exist $s \in S(\alpha)$, $\nu \in \{+, -\}$ and $\beta \in X_2 \setminus \{\alpha\}$ such that $|\Nnuo(s) \cap U(\beta)| \ge \phi_1 kt$. If $\alpha \prec \beta$ then by choice of $X_1$ there exists $\gamma \in Y_1$ such that $|\Nnuo(s) \cap U(\gamma)| \ge \phi_1 kt$. Similarly, if $\beta \prec \alpha$ then there exists $\gamma \in Y_2$ such that $|\Nnuo(s) \cap U(\gamma)| \ge \phi_1 kt$.
				\end{proof}

				Take $\C_1''$ to be a subset of $\C_1'$, chosen uniformly at random, and let $\C_1'''$ be the set of elements $\alpha$ in $\C_1'$ such that $|\Nnuo(s) \cap W(\C_1'' \setminus \{\alpha\})| \ge \phi_1 kt$ for every $s \in S(\alpha)$ and $\nu \in \{+, -\}$.

				\begin{claim} \label{claim:available-two}
					$\Pr[\alpha \in \C_1'''] \ge 1/2$ for every $\alpha \in \C_1'$.
				\end{claim}

				Using \Cref{claim:available-two}, which we prove below, we find that each of the events $\{\alpha \in \C_1''\}$ and $\{\alpha \in \C_1'''\}$ occurs with probability at least $1/2$, for every $\alpha \in \C_1'$. Noting that these events are independent for $\alpha \in \C_1'$, it follows that $\Ex(|\C_1'' \cap \C_1'''|) \ge |\C_1'|/4 \ge |\C_0|/36\rho^2$. Take an instance where the intersection $\C_1'' \cap \C_1'''$ has size at least $|\C_0|/36\rho^2$, and set $\C_1$ to be this intersection. So $\C_1$ has the required size and the desired property that $|\Nnuo(s) \cap W(\C_1 \setminus \{\alpha\})| \ge \phi_1 kt$ for every $\alpha \in \C_1$, $s \in S(\alpha)$ and $\nu \in \{+, -\}$.
			\end{proof}

			\begin{proof}[Proof of \Cref{claim:available-two}]
				Fix $\alpha \in \C_1'$, $s \in S(\alpha)$ and $\nu \in \{+, -\}$. Write $\C' := \C_1' \setminus \{\alpha\}$, $W' := W(\C')$, $W'' := W(\C_1'' \setminus \{\alpha\})$ and $N := \Nnuo(s)$.
				Let $E_{s, \nu}$ be the event that $|N \cap W''| \ge \phi_1 kt$. We will show that $\Pr[E_{s, \nu}] \ge 1 - \exp(-\phi_0 / 8\phi_1)$. This will imply that 
				\begin{equation*}
					\Pr[\alpha \in \C_1'''] 
					= \Pr\left[\bigcap_{s, \nu} E_{s, \nu}\right] 
					\ge 1 - 2|S(\alpha)| \cdot \exp(-\phi_0 / 8 \phi_1) 
					\ge 1 - 2\rho \exp(-8\rho) 
					\ge 1/2, 
				\end{equation*}
				as claimed (using $\phi_0 / \phi_1 = 2^6\rho$).

				We now estimate $\Pr[E_{s, \nu}]$.
				Recall that $|N| \ge \tau_1 kt / 2$. If $|N \cap W'| \ge \phi_1 kt$, then $\Pr[E_{s, \nu}] = 1$ (using $W' \subseteq W''$) so suppose this is not the case. 
				We claim that $|N \cap U(\beta)| \le \phi_1 kt$ for every $\beta \in \C'$.
				Indeed, if $|N \cap U(\beta)| \ge \phi_1 kt$ for some $\beta \in \C'$, then by choice of $\C_1'$ there exists $\gamma \in C_0 \setminus \C_1'$ such that $|N \cap U(\gamma)| \ge \phi_1 kt$. In particular, $|N \cap W'| \ge \phi_1 kt$, a contradiction to the previous assumption. 

				For each $\beta \in \C'$ write $m_{\beta} := |N \cap U(\beta)|$ and let $X_{\beta}$ be a random variable which is $m_{\beta}$ when $\beta \notin \C_1''$ and $0$ otherwise. Set $X := \sum_{\beta} X_{\beta}$, so that $X = |N \cap (W'' \setminus W')|$. By the previous paragraph, $m_{\beta} \le \phi_1 kt$ for every $\beta \in \C'$ and $\sum_{\beta} m_{\beta} = |N \setminus W'| \ge \tau_1 kt/2 - \phi_1 kt \ge \tau_1 kt / 4 = \phi_0 kt$.
				Thus, by \Cref{prop:hoeffding}, $\Pr[E_{s, \nu}] \ge \Pr[X \ge \phi_1 kt] \ge 1 - \exp(-\phi_0 / 8\phi_1)$, as claimed.
			\end{proof}

		\subsubsection{Proof of \Cref{lem:available-step-2}}
			The proof of \Cref{lem:available-step-2} will follow from two applications of the following lemma.

			\def \Mnu {M^{\nu}}
			\def \D {\mathcal{D}}
			\begin{lemma} \label{lem:available-step-2-modified}
				Let $\D_1 \subseteq \C_1$, $\mu \in \{+, -\}$ and $\theta_1 \le \phi_1$, and write $\theta_2 = \theta_1 / 2^7 \rho$. For each $\alpha \in \D_1$, $s \in S(\alpha)$ and $\nu \in \{+, -\}$, let $\Mnu(s)$ be a subset of $\Vo$ of size at least $\theta_1 kt$. Then there is a subset $\D_2 \subseteq \D_1$ of size at least $|\D_1|/36\rho^2$ such that for every $\alpha \in \D_2$, $s \in S(\alpha)$ and $\nu \in \{+, -\}$ there are at least $\theta_2 kt$ vertices $u \in \Mnu(s)$ such that $|\Nmu(u) \cap W(\D_2 \sm \{\alpha\})| \ge \theta_2 kt$.
			\end{lemma}

			The proof of \Cref{lem:available-step-2} follows easily from \Cref{lem:available-step-2-modified}. Indeed, first apply the latter lemma to $\C_1$ with $\mu = +$ and $\theta_1 = \phi_1$, where, for $\alpha \in \C_1$, we define $\Mnu(s) := \Nnuo(s) \cap W(\C_1 \setminus \{\alpha\})$; denote the resulting set $\C_2'$. Now apply the lemma again to $\C_2'$ with $\mu = -$ and $\theta_1 = \phi_1 / 2^7 \rho$, where for $\alpha \in \C_2'$ we take $\Mnu(s)$ to be the set of vertices $u$ in $\Nnuo(s) \cap W(\C_1 \setminus \{\alpha\})$ for which $|\Npg(u) \cap W(\C_2' \sm \{\alpha\})| \ge \theta_2 kt$ (note that $|\Mnu(s)| \ge \theta_1 kt$ for every $s \in S(\alpha)$ and $\nu \in \{+, -\}$ by the choice of $\C_2'$); take $\C_2$ to be the set resulting of the latter application. Then $|\C_2| \ge |\C_1| / 36^2 \rho^4$ and for every $\alpha \in \C_2$, $s \in S(\alpha)$ and $\nu \in \{+, -\}$, the set $\Nnuo(s) \cap W(\C_1 \setminus \{\alpha\})$ contains at least $(\phi_1 / 2^{14} \rho^2) kt = \phi_2 kt$ vertices $u$ satisfying $|\Nmu(u) \cap W(\C_2 \sm \{\alpha\})| \ge \phi_2 kt$ for $\mu \in \{+, -\}$, as required for \Cref{lem:available-step-2}.

			\begin{proof}[Proof of \Cref{lem:available-step-2-modified}]
				We proceed similarly to the proof of \Cref{lem:available-step-1}, first picking a subset $\D_2' \subseteq \D_1$ as in the following claim.
				\begin{claim} \label{claim:available-three}
					There is a subset $\D_2' \subseteq \D_1$ of size at least $|\D_1|/9\rho^2$ such that for every $\alpha \in \D_2'$, $s \in S(\alpha)$ and $\nu \in \{+, -\}$, if there exists $\beta \in \D_2' \setminus \{\alpha\}$ for which $\Mnu(s)$ contains at least $\theta_2 kt$ vertices $u$ with $|\Nmu(u) \cap U(\beta)| \ge \theta_2 kt$, then there exists $\gamma \in \D_1 \setminus \D_2'$ with the same property.
				\end{claim}

				\begin{proof}
					The proof is very similar to that of \Cref{claim:available-one}. We pick an ordering $\prec$ of $\C_1$ and find partitions $\{X_1, Y_1\}$ of $\C_1$ and $\{X_2, Y_2\}$ of $X_1$ as before, namely that after adding an element $\alpha$ to $X_1$, for each $s \in S(\alpha)$ and $\nu \in \{+, -\}$, if there exists $\beta$ in $\C_1 \setminus (X_1 \cup Y_1 \cup \{\alpha\})$ such that $\Mnu(s)$ contains at least $\theta_2 kt$ vertices $u$ with $|\Nmu(u) \cap U(\beta)| \ge \theta_2 kt$, then we move one such $\beta$ to $Y_1$.
					A similar modification is done when defining $\{X_2, Y_2\}$.
					Take $\D_2' := X_2$ and analyse as before.
				\end{proof}
				Take $\D_2''$ to be a random subset of $\D_2'$, chosen uniformly at random, and let $\D_2'''$ be the set of elements $\alpha \in \D_2'$ such that for every $s \in S(\alpha)$ and $\nu \in \{+, -\}$, the set $\Mnu(s)$ contains at least $\theta_2 kt$ vertices $u$ for which $|\Nmu(u) \cap W(\D_2'' \setminus \{\alpha\})| \ge \theta_2 kt$. 
				As above, the next claim implies that there is a suitable choice of $\D_2$ with $|\D_2| \ge |\D_2'|/4 \ge |D_1|/36\rho^2$.
			\end{proof}

			\begin{claim}
				$\Pr[\alpha \in \D_2'''] \ge 1/2$ for every $\alpha \in \D_2'$.
			\end{claim}

			\begin{proof}

				Fix $\alpha \in \D_2'$, $s \in S(\alpha)$ and $\nu \in \{+, -\}$. 
				We define the following notation.
				\begin{align*}
					\begin{array}{ll}
						\D' := \D_2' \setminus \{\alpha\}, 
						& \qquad \D'' := \D_2'' \setminus \{\alpha\}. \vspace{3pt}\\
						W' := W(\D'),
						& \qquad W'' := W(\D'').
					\end{array}
				\end{align*}
				Additionally, we define
				\begin{align*}
					M := \Mnu(s), 
					\qquad \qquad N'(u) := \Nmu(u) \text{ for $u \in M$}.
				\end{align*}
				Let $E$ be the event that $M$ has at least $\theta_2 kt$ vertices $u$ for which $|N'(u) \cap W''| \ge \theta_2 kt$. We will show that $\Pr[E] \ge 1 - \exp(-\theta_1/16\theta_2)$, which will suffice to prove the claim.

				\def \Mi {M_{\infty}}
				Define $f : M \to \D' \cup \{\infty\}$ as follows: for $u \in M$, if there is $\beta \in \D'$ such that $|N'(u) \cap U(\beta)| \ge \theta_2 kt$, set $f(u) := \beta$ (there may be several such $\beta$, pick one arbitrarily); otherwise, set $f(u) := \infty$. Define $\Mi := N \cap f^{-1}(\infty)$. We consider two cases: $|\Mi| \ge 2\theta_2 kt$ and $|\Mi| \le 2\theta_2 kt$.

				Consider the former case.
				Given $u \in M$, we claim that $|N'(u) \cap W''| \ge \theta_2 kt$ holds with probability at least $1 - 2\exp(-\theta_1/8\theta_2)$.
				Indeed, this event holds with probability $1$ if $|N'(u) \cap W'| \ge \theta_2 kt$, so we assume $|N'(u) \cap W'| \le \theta_2 kt$. Denote $m_{\beta} := |N'(u) \cap U(\beta)|$ for $\beta \in \D'$, and let $X_{\beta}$ be a random variable which takes value $m_{\beta}$ when $\beta \notin \D_2''$ and is $0$ otherwise. Set $X := \sum_{\beta} X_{\beta}$. 
				Then $m_{\beta} \le \theta_2 kt$ for every $\beta \in \D'$ (because $u \in \Mi$) and $\sum_{\beta} m_{\beta} = |N'(u) \setminus W'| \ge \tau_1 kt / 2 - \theta_2 kt \ge \theta_1 kt$ (by assumption on $|N'(u) \cap W'|$ and using $|N'(u)| \ge \tau_1 kt/2$, which holds for every $u \in \Vo$). By \Cref{prop:hoeffding} and definition of $X$, we have $|N'(u) \cap W''| \ge X \ge \theta_2 kt$ with probability at least $1 - \exp(-\theta_1/8\theta_2)$. It follows from \Cref{prop:markov} that, with probability at least $1 - \exp(-\theta_1/16\theta_2)$, the event $|N'(u) \cap W''| \ge \theta_2 kt$ holds for at least $\theta_2 kt$ values of $u \in \Mi$ (using the assumption $|\Mi| \ge 2\theta_2 kt$). This proves $\Pr[E] \ge 1 - \exp(-\theta_1/16\theta_2)$ in this case.

				Now consider the latter case, where $|\Mi| \le 2\theta_2 kt$.
				Then $|M \cap f^{-1}(\D')| \ge (\theta_1 - 2\theta_2)kt \ge \theta_1 kt / 2$. 
				If there is $\beta \in \D'$ which is the image of at least $\theta_2 kt$ vertices $u$ in $M$, then $\Pr[E] = 1$ by choice of $\D_2'$. 
				So suppose that every $\beta$ in $\D'$ is the image of at most $\theta_2 kt$ vertices $u$ in $M$. 
				Let $X_{\beta}$ be the random variable which is $|M \cap f^{-1}(\beta)|$ if $\beta \notin \D''$ and $0$ otherwise. Setting $X := \sum_{\beta} X_{\beta}$, we have $X \ge \theta_2 kt$ with probability at least $1 - \exp(-\theta_1 / 16 \theta_2)$, by \Cref{prop:hoeffding}. Note that $X$ lower bounds the number of vertices $u \in N$ for which $|N'(u) \cap W''| \ge \theta_2 kt$, as each vertex in $M$ is counted at most one time. 
				Thus, this implies $\Pr[E] \ge 1 - \exp(-\theta_1 / 16 \theta_2)$, as required.
			\end{proof}
		
		\subsubsection{Proof of \Cref{lem:available-step-3}}

			\def \Ma {M_{\alpha}}
			\def \Wa {W_{\alpha}}
			As above, we prove \Cref{lem:available-step-3} by twice applying the following lemma.
			\begin{lemma} \label{lem:available-step-3-modified}
				Let $\D_1 \subseteq \C_2$ and $\theta_1 \le \phi_2$, and write $\theta_2 = \theta_1 / 2^8 \rho$. For each $\alpha \in \D_1$, $s \in S(\alpha)$ and $\nu \in \{+, -\}$ let $\Mnu(s)$ be a set of size at least $\theta_1 kt$ such that each $u \in \Mnu(s)$ is associated with a subset $\Ma(u)$ of $\Vo$ of size at least $\theta_1 kt$. Then there is a subset $\D_2 \subseteq \D_1$ of size at least $|\D_1| / 36 \rho^2$ such that for every $\alpha \in \D_2$, $s \in S(\alpha)$ and $\nu \in \{+, -\}$ there are at least $\theta_2 kt$ vertices $u \in \Mnu(s)$ such that $\Ma(u)$ contains at least $\theta_2 kt$ vertices $v$ with $|\Nmg(v) \cap W(\D_2 \sm \{\alpha\})| \ge \theta_2 kt$.
			\end{lemma}

			Before proving \Cref{lem:available-step-3-modified}, we show that it implies \Cref{lem:available-step-3}. Write $\Wa := W(\C_2 \sm \{\alpha\})$ here for brevity. Indeed, first apply the former lemma to $\C_2$ with $\theta_1 = \phi_2$, where $\Ma(u) := \Npg(u) \cap \Wa$ and $\Mnu(s)$ is the set of vertices $u$ in $\Nnuo(s) \cap \Wa$ for which $|\Nmu(u) \cap \Wa| \ge \theta_1 kt$ for $\mu \in \{+, -\}$ (so $|\Mnu(s)| \ge \theta_1 kt$ for $\alpha \in \C_2$, $s \in S(\alpha)$ and $\nu \in \{+, -\}$). Denote the resulting set $\C_2'$ and apply the same lemma to $\C_2'$, with $\theta_1 = \phi_2 / 2^8 \rho$, where now $\Ma(u) := \Nmg(u) \cap \Wa$ and $\Mnu(s)$ is the set of vertices $u$ in $\Nnuo(s) \cap \Wa$ for which $|\Nmg(u) \cap \Wa| \ge \theta_1 kt$ and $\Npg(u) \cap \Wa$ contains at least $\theta_1 kt$ vertices $v$ with $|\Nmg(v) \cap W(\C_2' \sm \{\alpha\})| \ge \theta_1 kt$. Take $\C_3$ to be the set resulting from the latter application. Then $\C_3$ satisfies the requirements of \Cref{lem:available-step-3}.

			\begin{proof}
				We pick a subset $\D_2' \subseteq \D_1$ as in the following claim, which can be proved similarly to \Cref{claim:available-one} and \Cref{claim:available-three}. 
				\begin{claim}
					There exists $\D_2' \subseteq \D_1$ of size at least $|\D_1|/9\rho^2$ such that for every $\alpha \in \D_2'$, $s \in S(\alpha)$ and $\nu \in \{+, -\}$, if $\Mnu(s)$ contains at least $\theta_2 kt$ vertices $u$ such that $\Ma(u)$ contains at least $\theta_2 kt$ vertices $v$ with $|\Nmg(v) \cap U(\beta)| \ge \theta_2 kt$, for some $\beta \in \D_2' \setminus \{\alpha\}$, then there exists $\gamma \in \D_1 \setminus \D_2'$ with the same property.
				\end{claim}

				As usual, let $\D_2''$ be a random subset of $\D_2'$, chosen uniformly at random, and let $\D_2'''$ be the set of elements $\alpha \in \D_2'$ such that for every vertex $s \in S(\alpha)$ and $\nu \in \{+, -\}$, the set $\Mnu(s)$ contains at least $\theta_2 kt$ vertices $u$ for which $\Ma(u)$ contains at least $\theta_2 kt$ vertices $v$ that satisfy $|\Nmg(v) \cap W(\D_2'' \setminus \{\alpha\})| \ge \theta_2 kt$. Again, as usual, it suffices to prove the following claim.
			\end{proof}

			\begin{claim}
				$\Pr[\alpha \in \D_3'''] \ge 1/2$ for $\alpha \in \D_3'$.
			\end{claim}

			\begin{proof}
				Fix $\alpha \in \D_2'$, $s \in S(\alpha)$ and $\nu \in \{+, -\}$. We define the following notation.
				\begin{align*}
					\begin{array}{lll}
						\D := \D_1 \setminus \{\alpha\}, 
						& \qquad \D' := \D_2' \setminus \{\alpha\}, 
						& \qquad \D'' := \D_2'' \setminus \{\alpha\}. \\
						W := W(\D),
						& \qquad W' := W(\D'),
						& \qquad W'' := W(\D'').
					\end{array}
				\end{align*}
				Additionally, let $M''(v) := \Nmg(v)$, $M'(u) := \Ma(u)$ and $M := \Mnu(s)$.

				Let $E$ be the event that $M$ has at least $\theta_2 kt$ vertices $u$ for which $M'(u)$ contains at least $\theta_2 kt$ vertices $v$ such that $|M''(v) \cap W''| \ge \theta_2 kt$. As before, it suffices to prove $\Pr[E] \ge 1 - \exp(-\theta_1/32\theta_2)$.
				 
				Define functions $f_1, f_2, f_3, f_4$ from subsets of $V$ to $\D' \cup \{\infty\}$, as follows.
				\begin{enumerate}
					\item
						For $w \in V$, if $w \in U(\beta)$ with $\beta \in \D'$, put $f_1(w) := \beta$; otherwise $f_1(w) := \infty$.
					\item
						For $v \in \Vo$, if there exists $\beta \in \D'$ such that $f_1(w) = \beta$ for at least $\theta_2 kt$ vertices $w$ in $M''(v)$, put $f_2(v) := \beta$ (there may be several suitable choices of $\beta$, pick one arbitrarily); otherwise $f_2(v) := \infty$.
					\item
						For $u \in M$, if there exists $\beta \in \D'$ such that $f_2(v) = \beta$ for at least $\theta_2 kt$ vertices $v$ in $M'(u)$, put $f_3(u) := \beta$; otherwise $f_3(u) := \infty$. 
					\item
						If there exists $\beta \in \D'$ such that $f_3(u) = \beta$ for at least $\theta_2 kt$ vertices $u$ in $M$, put $f_4(s) := \beta$; otherwise $f_4(s) := \infty$.
				\end{enumerate}

				We draw the following conclusions regarding the above functions.

				\begin{itemize}
					\item
						Let $v \in \Vo$ satisfy $f_2(v) = \infty$.
						Then $|M''(v) \cap W''| \ge \theta_2 kt$ with probability at least $1 - \exp(-\theta_1 / 8\theta_2)$, by \Cref{prop:hoeffding} (using $|M''(v)| \ge \tau_1 kt / 2$).
					\item
						Let $u \in M$ satisfy $f_3(u) = \infty$. 
						We claim that $M'(u)$ contains at least $\theta_2 kt$ vertices $v$ with $|M''(v) \cap W''| \ge \theta_2 kt$, with probability at least $1 - \exp(-\theta_1/16\theta_2)$.
						Indeed, if $|M'(u) \cap f_2^{-1}(\infty)| \ge 2\theta_2 kt$, this follows from \Cref{prop:markov} and the previous item, and, otherwise, it follows from \Cref{prop:hoeffding}, using $|M'(u)| \ge \theta_1 kt$.
					\item
						Suppose that $f_4(s) = \infty$.
						We claim that, with probability at least $1 - \exp(-\theta_1/32\theta_2)$, the set $M$ contains at least $\theta_2 kt$ vertices $u$ for which $M'(u)$ contains at least $\theta_2 kt$ vertices $v$ with $|M''(v) \cap W''| \ge \theta_2 kt$.
						If $|M \cap f_3^{-1}(\infty)| \ge 2\theta_2 kt$, then the claim follows from \Cref{prop:markov} and the previous item. Otherwise, it follows from \Cref{prop:hoeffding} (using $|M| \ge \theta_1 kt$).
				\end{itemize}
				In particular, if $f_4(s) = \infty$ then $\Pr[E] \ge 1 - \exp(-\theta_1/32\theta_2)$, and, otherwise, $\Pr[E] = 1$, by choice of $\D_2'$. 
			\end{proof}

	\subsection{Partition with many eligible neighbours} \label{subsec:eligible}

		In this subsection, we obtain a collection of pairwise disjoint sets of $10k$ gadgets, such that every vertex in each of these gadgets has many out- and in-neighbours that are candidates for connecting the gadgets in its set. Here is a definition of such candidates, and in \Cref{prop:eligible} we state the main result of this section, which proves the existence of a collection as described.

		For subsets $\A \subseteq \A_2$ and $W \subseteq V$, we say that a vertex $u$ is \emph{eligible} for $\A$ in $W$ if one of the following holds. 
		\begin{enumerate}[label = \rm(E\arabic*)]
			\item \label{itm:eligible-good}
				$u \in \Vg \cap W$ and $u$ has an out-neighbour in all but at most $k$ sets $\Sm(\alpha)$ with $\alpha \in \A$ and an in-neighbour in all but at most $k$ sets $\Sp(\alpha)$ with $\alpha \in \A$.
			\item \label{itm:eligible-out-good}
				$u \in (\Vgp \setminus \Vgm) \cap W$ and $u$ has an out-neighbour in all but at most $k$ set $\Sm(\alpha)$ with $\alpha \in \A$ and at least $\tau_3 kt$ in-neighbours that satisfy \ref{itm:eligible-good}.
			\item \label{itm:eligible-in-good}
				$u \in (\Vgm \setminus \Vgp) \cap W$ and $u$ has an in-neighbour in all but at most $k$ sets $\Sp(\alpha)$ with $\alpha \in \A$ and at least $\tau_3 kt$ out-neighbours that satisfy \ref{itm:eligible-good} or \ref{itm:eligible-out-good}.
			\item \label{itm:eligible-bad}
				$u \in \Vb \cap W$ and $u$ has at least $\tau_3 kt$ in-neighbours that satisfy \ref{itm:eligible-good} and at least $\tau_3 kt$ out-neighbours that satisfy \ref{itm:eligible-good} or \ref{itm:eligible-out-good}.
		\end{enumerate}
		We remark that this definition is similar to that of an available vertex, with the main difference being that whenever a vertex is required to have out-neighbours in all but at most $kt$ sets $S(\alpha)$ with $\alpha \in \A$ for it to be available for $\A$, here it is required to have out-neighbours in all but at most $k$ sets $S(\alpha)$ (and similarly for in-neighbours). 

		Recall that $\sp(\alpha)$ and $\smm(\alpha)$ are vertices in $S(\alpha)$ and $X(\alpha)$ is a set of size at most $\rho\sigma_1 kt$ such that all out-neighbours of $\sp(\alpha)$ in $V \setminus X(\alpha)$ are also out-neighbours of every vertex in $U(\alpha) \setminus S(\alpha)$, and similarly for in-neighbours of $\smm(\alpha)$.

		\begin{proposition} \label{prop:eligible}
			There is a subset $\A_3 \subseteq \A_2$ and a partition $\{\B_1, \ldots, \B_{\sigma_3 t}\}$ of $\A_3$, such that 
			\begin{itemize}
				\item
					$|\B_i| = 10k$ for every $i \in [\sigma_3 t]$.
				\item
					 For every $i \in [\sigma_3 t]$, $\alpha \in \B_i$ and $u \in S(\alpha)$, the vertex $u$ has at least $\tau_3 kt$ out- and in-neighbours that are eligible for $\B_i$ in $W(\A_3 \setminus \B_i) \sm X(\alpha)$.
				\item
					For every $i \in [\sigma_3 t]$, there are at least $n/10$ vertices in $\Vg \cap W(\A_3)$ that are eligible for $\B_i$.
			\end{itemize}
		\end{proposition}

		\begin{proof}
			\def \Wa {W_{\alpha}}
			Let $\phi = 16\sigma_3$; so $\phi = 160$.
			We partition $\A_2$ into $\phi t$ sets $\{\C_1, \ldots, \C_{\phi t}\}$, uniformly at random. For $\alpha \in \A_2$, let $i(\alpha)$ be the (random) index in $[\phi t]$ such that $\alpha \in \C_{i(\alpha)}$, and write $\Wa := W(\A_2 \setminus C_{i(\alpha)}) \setminus X(\alpha)$.

			\begin{claim} \label{claim:vertex-eligible}
				Let $\alpha \in \A_2$ and let $u$ be a vertex which is available for $\alpha$ (with respect to $\A_1$). Then $\Pr[\text{$u$ is eligible for $C_{i(\alpha)}$ in $\Wa$}] \ge 1 - 3\exp(-k/12\phi)$.
			\end{claim}

			\begin{proof}
				Write $\C := \C_{i(\alpha)}$ and $W := \Wa$. 
				We consider four cases, according to the four possible scenarios in the definition of an available vertex.

				First, suppose that $u \in \Vg \cap W$; so $u$ has in-neighbours in all but at most $kt$ sets $\Sp(\beta)$ with $\beta \in \A_2$, and similarly for out-neighbours in $\Sm(\beta)$. Denote the set of elements $\beta \in \A_2 \setminus \{\alpha\}$ for which $u$ does not have an in-neighbour in $\Sp(\beta)$ or an out-neighbour in $\Sm(\beta)$ by $\X$; so $|\X| \le 2kt$, and the expected size of $\C \cap \X$ is at most $2k/\phi$. It follows from Chernoff's bounds (see \Cref{lem:chernoff}) that $|\C \cap \X| \le 4k / \phi \le k$ with probability at least $1 - \exp(-2k/3\phi)$. In particular, $u$ is eligible for $\C$ (in $W$) with probability at least $1 - \exp(-2k/3\phi)$. 

				Next, suppose that $u \in (\Vgp \setminus \Vgm) \cap W$, so $u$ has out-neighbours in all but at most $kt$ sets $\Sm(\beta)$ and it has at least $\tau_2 kt - |X(\alpha)| \ge \tau_2 kt/2$ in-neighbours in $\Vg \cap W$ (using $\tau_2 \gg \rho, \sigma_1$); denote this set of in-neighbours by $N$. First observe that every element in $N$ is eligible for $\alpha$ with probability at least $1 - \exp(-2k/3\phi)$, by the previous paragraph. By \Cref{prop:markov}, with probability at least $1 - \exp(-k/3\phi)$, at least $(1 - \exp(-k/3\phi))|N| \ge \tau_3 kt$ elements of $N$ are eligible for $\alpha$ (using $\tau_2 \gg \tau_3$ and that $k$ is large). Similarly to the previous paragraph, with probability at least $1 - \exp(-2k/3\phi)$ the vertex $u$ has out-neighbours in all but at most $k$ sets $\Sm(\beta)$ with $\beta \in \C$. It follows that $u$ is eligible for $\alpha$ with probability at least $1 - 2\exp(-k/3\phi)$. 

				The next case is when $u \in (\Vgm \setminus \Vgp) \cap W$. Then $u$ has in-neighbours in all but at most $kt$ sets $\Sp(\beta)$ and it has at least $\tau_2 kt/2$ out-neighbours as in \ref{itm:available-good} or \ref{itm:available-out-good}. 
				Similarly to the above, with probability at least $1 - 2\exp(-k/6\phi)$, at least a $(1 - 2\exp(-k/6\phi))$-fraction of $u$'s out-neighbours that satisfy \ref{itm:available-good} or \ref{itm:available-out-good} are eligible for $u$. Also, $u$ has in-neighbours in all but at most $k$ sets $\Sp(\beta)$, with probability at least $1 - \exp(-2k/3\phi)$. Altogether it follows that $u$ is eligible for $\alpha$ with probability at least $1 - 3\exp(-k/6\phi)$.

				Finally, if $u$ is as in \ref{itm:available-bad}, then, using arguments as in the previous two paragraphs, with probability at least $1 - 3\exp(-k/12\phi)$ it has at least $\tau_3 kt$ out- and in-neighbours that are eligible for $\alpha$. 
			\end{proof}

			We claim that the following three events hold simultaneously with positive probability.
			\begin{enumerate}[label = \rm(\alph*)]
				\item \label{itm:eligible-many-neighs-element}
					For all but at most $t$ values of $\alpha$ in $\A_2$, every $u \in S(\alpha)$ has at least $\tau_3 kt$ out- and in-neighbours that are eligible for $\C_{i(\alpha)}$ in $\Wa$.
				\item \label{itm:eligible-equal-parts}
					All but at most $t$ values $i \in [\phi t]$ satisfy $|\C_i| \ge 10k$.
				\item \label{itm:eligible-many-neighs-part}
					For all but at most $6t$ values of $i$ in $[\phi t]$, there are at least $n/4$ vertices that are eligible for $\C_i$ in $W(\C_i) \cap \Vg$.
			\end{enumerate}

			Fix $\alpha \in \A_2$ and $s \in S(\alpha)$. Recall that by choice of $\A_2$, the vertex $s$ has at least $\tau_2 kt$ out- and in-neighbours in $W(\A_2 \setminus \{\alpha\})$ that are available for $\alpha$ (with respect to $\A_2$), and so the number of out- and in-neighbours of $s$ in $\Wa$ that are available for $\alpha$ is at least $(\tau_2 - \rho \sigma_1)kt \ge \tau_2 kt / 2$ (using $W(\A_2 \setminus \{\alpha\}) \sm \Wa \subseteq X(\alpha)$). Thus, by \Cref{claim:vertex-eligible} and \Cref{prop:markov}, $s$ has at least  $(1 - 2\exp(-k/24\phi)) \tau_2 kt / 2 \ge \tau_3 kt$ out-neighbours that are eligible for $\C_{i(\alpha)}$ in $\Wa$, with probability at least $1 - 2\exp(-k/24\phi)$. A similar statement holds for in-neighbours of $s$.
			By a union bound, every $s \in S(\alpha)$ has at least $\tau_3 kt$ out- and in-neighbours that are eligible for $\C_{i(\alpha)}$ in $\Wa$, with probability at least $1 - 8\rho  \exp(-k/24\phi)$. 
			It follows from \Cref{prop:markov} that, with probability at least $1 - 3\rho\exp(-k/48\phi) > 1/2$, for at least $(1 - 3\rho\exp(-k/48\phi))|\A_2| \ge |\A_2| - t$ values of $\alpha$ in $\A_2$, every $s \in S(\alpha)$ has at least $\tau_3 kt$ out- and in-neighbours that are eligible for $\C_{i(\alpha)}$ in $\Wa$. Thus \ref{itm:eligible-many-neighs-element} holds with probability larger than $1/2$.

			By Chernoff's bounds (\Cref{lem:chernoff}), $|\C_i| \ge \sigma_2 k / 2\phi \ge 10k$ (using $\phi = 160$ and $\sigma_2 = 10^4$) with probability at least $1 - \exp(-\sigma_2 k/8\phi)$, for every $i \in [\phi t]$. It follows from \Cref{prop:markov} that $|\C_i| \ge 10k$ for all but at most $\exp(-\sigma_2 k / 16\phi) \phi t \le t$ values of $i \in [\phi t]$, with probability at least $1 - 2\exp(-\sigma_2 k / 16\phi) > 1/2$. So \ref{itm:eligible-equal-parts} holds with probability larger than $1/2$.

			Let $H$ be a bipartite auxiliary graph, with parts $\Vg$ and $[\phi t]$, such that $ui$ (with $u \in \Vg$ and $i \in [\phi t]$) is an edge in $H$ if $u$ has no in-neighbours in at least $k$ sets $\Sp(\alpha)$ with $\alpha \in \C_i$; or $u$ has no out-neighbours in at least $k$ sets $\Sm(\alpha)$ with $\alpha \in \C_i$; or $u \in U(\C_i)$. Then $e(H) \le 2t \cdot |\Vg| + n \le 3t \cdot |\Vg|$ (using that $|\Vg| \ge n/2$; see \ref{itm:gadget-good}). Thus all but at most $6t$ values of $i$ in $[\phi t]$ satisfy $d_H(i) \le |\Vg|/2$. Given such $i$, at least $|\Vg|/2 \ge n/4$ vertices $u \in \Vg$ are in $W(\C_i)$ and are eligible for $\C_i$.
			This shows that \ref{itm:eligible-many-neighs-part} holds deterministically.

			To summarise, the above three properties hold simultaneously with positive probability. Fix an instance of $\{\C_1, \ldots, \C_{\phi t}\}$ for which \ref{itm:eligible-many-neighs-element}, \ref{itm:eligible-equal-parts} and \ref{itm:eligible-many-neighs-part} all hold. 
			Let $I$ be the set of elements $i$ in $[\phi t]$ such that: for every $\alpha \in \C_i$, every $s \in S(\alpha)$ has at least $\tau_3 kt$ out- and in-neighbours in $\Wa$ that are eligible for $\C_i$; $|\C_i| \ge 10k$; at least $n/4$ vertices in $W(\C_i) \cap \Vg$ are eligible for $\C_i$. Then $|I| \ge \phi t - 8t \ge \phi t / 2$. Let $J$ be a subset of $I$ that consists of the $|I|/8 \ge \phi t / 16 = \sigma_3 t$ elements $i$ in $I$ for which $U(\C_i)$ is smallest; so $|\bigcup_{j \in J} U(\C_j)| \le n/8$. For each $j \in J$, let $\C_j'$ be a subset of $\C_j$ of size exactly $10k$. Let $j_1, \ldots, j_{\sigma_3 t} \in J$ be distinct, let $\B_i := \C_{j_i}'$ for $i \in [\sigma_3 t]$, and set $\A_3 := \bigcup_i \B_i$.  

			The sets $\A_3, \B_1, \ldots, \B_{\sigma_3 t}$ satisfy the requirements of the proposition. 
			Indeed, clearly $|\B_i| = 10k$ for $i \in [\sigma_3 t]$. The second bullet holds because $W(\A_3 \setminus \B_i) \setminus X(\alpha) \subseteq W(\A_2 \setminus \C_{j_i}) \setminus X(\alpha) = \Wa$ for $i \in [\sigma_3 t]$ and $\alpha \in \B_i$.
			Finally, because $|U(\A_3)| \le n/8$ and by \ref{itm:eligible-many-neighs-part}, there are at least $n/8$ vertices in $\Vg \cap W(\A_3 )$ that are eligible for $\B_i$, for $i \in [\sigma_3 t]$.
		\end{proof}

	\subsection{Many connected sets} \label{subsec:connected}

		We now find a collection of $t$ pairwise disjoint $k$-connected sets with some additional properties that will later allow us to add the unused vertices into these sets while maintaining $k$-connectivity.

		\begin{proposition} \label{prop:connected-parts}
			There exist pairwise-disjoint sets $V_1, \ldots, V_{t}$ such that 
			\begin{itemize}
				\item
					$2k \le |V_i| \le n/t$ for $i \in [t]$.
				\item
					$T[V_i]$ is $k$-connected for $i \in [t]$.
				\item
					Each $V_i$ contains at least $10k$ sets $U(\alpha)$ with $\alpha \in \A_3$.
				\item
					Let $Z := V \setminus (V_1 \cup \ldots \cup V_{t})$. For every $i \in [t]$, there are at least $n/100$ vertices in $Z \cap \Vg$ that have at least $k$ out- and in-neighbours in $V_i$.
			\end{itemize}
		\end{proposition}

		The next notion, of \emph{helpful} vertices, will aid us in the proof of \Cref{prop:connected-parts}.
		Given subsets $\A \subseteq \A_3$ and $W \subseteq W(\A_3 \setminus \A)$, we say that a vertex $u$ is \emph{helpful} for $\A$ in $W$ if one of the following holds.
		\begin{enumerate}[label = \rm(H\arabic*)]
			\item \label{itm:helpful-good}
				$u \in \Vg \cap W$ and $u$ has out-neighbours in all but at most $k$ sets $\Sm(\alpha)$ and in-neighbours in all but at most $k$ sets $\Sp(\alpha)$ with $\alpha \in \A$.
			\item \label{itm:helpful-out-good}
				$u \in (\Vgp \setminus \Vgm) \cap W$ and $u$ has out-neighbours in all but at most $k$ sets $\Sm(\alpha)$ with $\alpha \in \A$ and at least $k$ in-neighbours as in \ref{itm:helpful-good}.
			\item \label{itm:helpful-in-good}
				$u \in (\Vgm \setminus \Vgp) \cap W$ and $u$ has in-neighbours in all but at most $k$ sets $\Sp(\alpha)$ with $\alpha \in \A$ and at least $k$ out-neighbours as in \ref{itm:helpful-good} or \ref{itm:helpful-out-good}.
			\item \label{itm:helpful-bad}
				$u \in \Vb \cap W$ and $u$ has at least $k$ in-neighbours as in \ref{itm:helpful-good} and at least $k$ out-neighbours as in \ref{itm:helpful-good} or \ref{itm:helpful-out-good}.
		\end{enumerate}

		The following claim illustrates how helpful vertices can be used along with sets $U(\A)$ to form $k$-connected sets.

		\begin{claim} \label{claim:helpful}
			Let $\A \subseteq \A_3$ be a set of size at least $3k$, and let $W \subseteq W(\A_3 \sm \A)$. Suppose every vertex in $S(\alpha)$ has at least $k$ out- and in-neighbours in $W \setminus X(\alpha)$ that are helpful for $\A$ in $W$, for every $\alpha \in \A$. Then there is a set $Y$ such that $U(\A) \subseteq Y \subseteq W \cup U(\A)$ and $T[Y]$ is $k$-connected.
		\end{claim}

		\begin{proof}
			Let $W'$ be the set of vertices in $W$ that are helpful for $\A$ in $W$. Set $Y := W' \cup U(\A)$. We will show that $T[Y]$ is $k$-connected.
			To do this, we need to show that for every subset $Z \subseteq Y$ of size at most $k-1$ and every $u, v \in Y \setminus Z$, there is a directed path in $T[Y \setminus Z]$ from $u$ to $v$; fix such $Z, u, v$. Let $\A' \subseteq \A$ be the set of elements $\alpha$ in $\A$ such that $U(\alpha)$ is disjoint of $Z$; so $|\A'| > 2k$. 

			First suppose that $u, v \in W'$. It is easy to check, by definition of helpfulness and by choice of $\A'$, that there is a path in $Y \setminus Z$ from $u$ to all but at most $k$ sets $\Sm(\alpha)$ with $\alpha \in \A'$, and from all but at most $k$ sets $\Sp(\alpha)$ with $\alpha \in \A'$ to $v$. In particular, there exists $\alpha \in \A'$ such that there is a path from $u$ to $\Sm(\alpha)$ and from $\Sp(\alpha)$ to $v$. Since $U(\alpha)$ is disjoint of $Z$, and there is a path from any element of $\Sm(\alpha)$ to any element of $\Sp(\alpha)$ in $U(\alpha)$, it follows that there is a path from $u$ to $v$ in $T[Y']$, as required.

			Next, if $u \notin W'$, then $u \in U(\alpha)$ for some $\alpha \in \A$. Since every vertex in $S(\alpha)$ has at least $k$ out-neighbours in $W' \setminus X(\alpha)$, the same holds for $u$ by choice of $X(\alpha)$ (see \ref{itm:gadget-neighbourhood}), and so $u$ has an out-neighbour $u'$ in $W' \setminus Z$. Similarly, if $v \notin W'$ then $v$ has an in-neighbour $v'$ in $W' \setminus Z$ (if $v \in W'$, put $v' = v$). Since, by the previous paragraph, there is a path from $u'$ to $v'$ in $Y \setminus Z$, it follows that there is a path from $u$ to $v$ in $Y \setminus Z$, as required.
		\end{proof}

		The next claim will allow us to find many helpful vertices for a subset $\A \subseteq \A_3$. 

		\begin{claim} \label{claim:vertex-helpful}
			Let $\A \subseteq \A_3$, let $Z$ be a random subset of $W(\A_3) \cap \Vo$, obtained by picking each vertex with probability at least $1/100t$, independently. Then every vertex $u \in W(\A_3 \setminus \A)$, which is eligible for $\A$, is helpful for $\A$ in $Z \cup \{u\}$, with probability at least $1 - 3\exp(-\tau_3 k / 10^5)$. 
		\end{claim}

		\begin{proof}
			Let $W'$ be the set of vertices in $W(\A_3) \cap \Vo$ that are eligible for $\A$, and let $Z' := Z \cap W'$. We think of $Z'$ as the union of $Z_1$, $Z_2$, $Z_3$ and $Z_4$, which are random subsets of $\Vg \cap W'$, $(\Vgp \setminus \Vgm) \cap W'$, $(\Vgm \setminus \Vgp) \cap W'$ and $\Vb \cap W'$, respectively, each obtained by including each potential vertex with probability $1 / 100t$, independently. 
			We consider four cases, according to the definition of a helpful vertex. It will be useful to note that if $u$ is helpful for $\A$ in a set $S$, then it is helpful for $\A$ in any set that contains $S$. 

			If $u \in \Vg \cap W'$, then $u$ is helpful for $\A$ in $\{u\}$. In particular, $u$ is helpful for $\A$ in $Z' \cup \{u\}$ with probability $1$.

			If $u \in (\Vgp \setminus \Vgm) \cap W'$, then $u$ has at least $\tau_3 kt$ in-neighbours in $\Vg \cap W'$. By Chernoff, at least $\tau_3 k / 10^3 \ge k$ of them are in $Z_1$, with probability at least $1 - \exp(-\tau_3 k / 10^3)$. In particular, $u$ is helpful in $Z_1 \cup \{u\}$, and thus in $Z' \cup \{u\}$, with probability at least $1 - \exp(-\tau_3 k / 10^3)$. 

			Next, suppose that $u \in (\Vgm \setminus \Vgp) \cap W'$, so $u$ has at least $\tau_3 kt$ out-neighbours in $\Vgp \cap W'$; denote the set of such out-neighbours by $N$, and write $N_1 = N \cap \Vg$ and $N_2 = N \setminus N_1$. If $|N_1| \ge |N|/2$, then by Chernoff, with probability at least $1 - \exp(-\tau_3 k / 10^4)$, at least $|N_1|/10^3t \ge \tau_3 k / 2000 \ge k$ of the vertices in $N_1$ are in $Z_1$, and so $u$ is helpful in $Z_1 \cup \{u\}$, with probability at least $1 - \exp(-\tau_3 k / 10^4)$. Now suppose that $|N_2| \ge |N|/2$. By the previous paragraph, every $v \in N_2$ is helpful for $\A$ in $Z_1 \cup \{v\}$ with probability at least $1 - \exp(-\tau_3 k / 10^3)$. Thus, by \Cref{prop:markov}, with probability at least $1 - \exp(-\tau_3 k / 10^4)$, at least $(1 - \exp(-\tau_3 k / 10^4))|N_2| \ge \tau_3 kt / 4$ vertices $v$ in $N_2$ are helpful in $Z_1 \cup \{v\}$. Conditioning on $Z_1$ satisfying this property, by Chernoff, there are at least $\tau_3 k / 10^4 \ge k$ vertices in $N_2$ that are in $Z_2$ and are helpful in $Z_1 \cup Z_2$, with probability at least $1 - \exp(-\tau_3 k / 10^4)$. It follows that $u$ is helpful in $Z_1 \cup Z_2 \cup \{u\}$, with probability at least $1 - 2\exp(-\tau_3 k / 10^4)$. 

			Finally, suppose that $u \in \Vb$. Then $u$ has at least $\tau_3 kt$ out-neighbours in $\Vgp \cap W'$ and at least $\tau_3 kt$ in-neighbours in $\Vg \cap W'$; denote these sets of out- and in-neighbours by $\Np$ and $\Nm$. By Chernoff, at least $\tau_3 k / 10^3$ vertices in $\Nm$ are in $Z_1$, with probability at least $1 - \exp(\tau_3 k / 10^3)$. By the previous paragraph and \Cref{prop:markov}, at least $k$ vertices in $\Np$ are helpful in $Z_1 \cup Z_2 \cup Z_3$, with probability at least $1 - 2\exp(\tau_3 k / 10^5)$. It follows that $u$ is helpful in $Z' \cup \{u\}$ with probability at least $1 - 3\exp(-\tau_3 k / 10^5)$.

			To summarise, we showed that every $u \in W'$ is helpful in $Z' \cup \{u\}$ with probability at least $1 - 3\exp(-\tau_3 k / 10^5)$, where $Z'$ is a random subset of $W'$ obtained by including each vertex with probability $1/100t$. The same thus holds for $Z$, which is a random subset of $W(\A_3) \cap \Vo$, obtained by including each vertex with probability at least $1/100t$, because we can couple the two random sets so that $Z' \subseteq Z$.
		\end{proof}

		\begin{proof} [Proof of \Cref{prop:connected-parts}]
			Let $\B_1, \ldots, \B_{10t}, \A_3$ be the sets produced by \Cref{prop:eligible} (recall that $\sigma_3 = 10$).
			Let $W := W(\A_3)$. We partition $W$ into sets $Y, W_1, \ldots, W_{10t}$ by putting each vertex in $Y$ with probability $1/2$, and otherwise in one of $W_1, \ldots, W_{10t}$, uniformly at random, independently.
			We claim that the following properties hold simultaneously with positive probability.
			\begin{enumerate}[label = \rm(\alph*)]
				\item \label{itm:helpful-reservoir}
					For every $i \in [10t]$, there are at least $n / 100$ vertices in $Y \cap \Vg$ that have at least $k$ out- and in-neighbours in $U(\B_i)$.
				\item \label{itm:helpful-connected}
					For at least $5t$ values of $i$, there is a set $Z_i$ such that $U(\B_i) \subseteq Z_i \subseteq U(\B_i) \cup W_i$ and $T[Z_i]$ is $k$-connected.
			\end{enumerate}
			Fix $i \in [10t]$. By assumption on $\B_i$, there is a set $W_e \subseteq \Vg \cap W$ of size at least $n/10$ whose vertices are eligible for $\B_i$, namely they have at least $k$ out- and in-neighbours in $U(\B_i)$. By Chernoff, $|W_e \cap W| \ge n / 100$ with probability at least $1 - \exp(-n/100)$. Thus, by a union bound, \ref{itm:helpful-reservoir} holds with probability at least $3/4$.

			Fix $i \in [10t]$, $\alpha \in \B_i$ and $s \in S(\alpha)$. Let $N$ be the set of out-neighbours of $s$ in $(W \cup U(\alpha)) \setminus X(\alpha)$ that are eligible for $\B_i$; so $|N| \ge \tau_3 kt$ by the assumptions (see \Cref{prop:eligible}).
			Let $p \in (0,1)$ satisfy $(1 - p)^2 = 1 - 1/20t$; so $p \ge 1/40t$. 
			Let $W_i'$ and $W_i''$ be two random sets, obtained by including each vertex in $W$ with probability $p$, independently. Notice that every vertex in $W$ is in $W_i' \cup W_i''$ with probability $1/20t$, so we may generate the sets $Y, W_1, \ldots, W_{10t}$ by taking $W_i$ to be $W_i' \cup W_i''$, then put each vertex in $W \setminus W_i$ in $Y$ independently with probability $(1 - 1/20t)/2$, and then put each of the remaining uncovered vertices in one of $W_1, \ldots, W_{i-1}, W_{i+1}, \ldots, W_{10t}$, independently and uniformly at random.

			By \Cref{claim:vertex-helpful} and \Cref{prop:markov}, the set $N'$ of vertices $u$ in $N$ that are helpful for $\B_i$ in $W_i' \cup \{u\}$ has size at least $|N|/2$, with probability at least $1 - 2\exp(-\tau_3 k / 10^6)$. Conditioning on this occurring, by Chernoff, $|N' \cap W_i''| \ge |N|/10^3 t \ge \tau_3 k / 10^3 \ge k$ with probability at least $1 - \exp(-|N| / 10^3 t) \ge 1 - \exp(-\tau_3 k / 10^3)$. It follows that $s$ has at least $k$ out-neighbours that are helpful for $\B_i$ in $(W_i \cup U(\alpha)) \setminus X(\alpha)$, with probability at least $1 - 3\exp(-\tau_3 k/10^5)$, and similarly for in-neighbours of $s$.
			A union bound, combined with the assumption that $|\B_i| = 10k$, shows that, with probability at least $1 - 60\rho k \exp(-\tau_3k / 10^5)$, for every $\alpha \in \B_i$, every $s \in S(\alpha)$ has at least $k$ out- and in-neighbours that are helpful for $\B_i$ in $(W_i \cup U(\alpha)) \setminus X(\alpha)$. 
			
			It follows from \Cref{prop:markov} that the latter property holds for at least $5t$ values of $i \in [10t]$, with probability at least $3/4$. 
			By \Cref{claim:helpful}, property \ref{itm:helpful-connected} holds with probability at least $3/4$. 

			This completes the proof that \ref{itm:helpful-reservoir} and \ref{itm:helpful-connected} hold simultaneously with positive probability. It is now easy to prove \Cref{prop:connected-parts}. Indeed, fix an instance of $Y, W_1, \ldots, W_{10t}$ such that \ref{itm:helpful-reservoir} and \ref{itm:helpful-connected} hold. Let $I$ be the set of indices $i \in [10t]$ for which \ref{itm:helpful-connected} holds (so $|I| \ge 5t$). For $i \in I$, let $Z_i$ be as in \ref{itm:helpful-connected}. Note that $|Z_i| \le n/t$ for all but at most $t$ indices $i \in I$. It follows that there are distinct $i_1, \ldots, i_{t} \in I$ such that $|Z_{i_j}| \le n/t$ for $j \in [t]$. Take $V_j := Z_{i_j}$. It is easy to see that the properties in \Cref{prop:connected-parts} hold.
		\end{proof}

	\subsection{Partition into connected sets} \label{subsec:partition}

		We are almost ready to complete the proof of \Cref{thm:main}.
		The following proposition does most of the remaining work.

		\begin{proposition} \label{prop:extending-partition}
			Let $Z, Y, V_1 \ldots, V_{t}$ be disjoint subsets of $V$ that satisfy the following properties.
			\begin{itemize}
				\item
					$|V_i| \ge 2k$ for $i \in [t]$.
				\item
					$T[V_i]$ is $k$-connected for $i \in [t]$.
				\item
					Every vertex in $Y$ has at least $k$ out- and in-neighbours in at least $3t/4$ sets $V_i$ with $i \in [t]$.
				\item
					For every $i \in [t]$, there are at least $n/10^3$ vertices in $Y$ that have at least $k$ out-\ and in-neighbours in $V_i$.
				\item
					Every vertex in $Z$ either has at least $k$ out-neighbours in $V_i$ for at least $3t/4$ values of $i \in [t]$, or has at least  $\max\{10^{10}|Z|, 100kt\}$ out-neighbours in $Y \cup V_1 \cup \ldots \cup V_{t}$; similarly for in-neighbours.
			\end{itemize}

			Then there is a partition $\{V_1', \ldots, V_{t}'\}$ of $Z \cup Y \cup V_1 \cup \ldots \cup V_{t}$ such that $V_i \subseteq V_i'$ and $T[V_i']$ is $k$-connected, for $i \in [t]$. 
		\end{proposition}
		
		\begin{proof}
			For $y \in Y$, let $I(y)$ be the set of elements $i \in [t]$ such that $y$ has at least $k$ out- and in-neighbours in $V_i$; so $|I(y)| \ge 3t/4$ for every $y \in Y$. Let $I_{\ell}(y)$ be the set of indices $i \in I(y)$ with $|V_i| \ge \frac{nk}{10^{10}\log n}$ and write $I_s(y) = I(y) \setminus I_{\ell}(y)$.
			We form a random partition $\{Y_1, \ldots, Y_{t}\}$ of $Y$ as follows: for each $y \in Y$, independently, put $y$ in $Y_i$ for some $i \in I(y)$ so that the following holds.
			\begin{equation*}
				\Pr\left(y \in Y_i\right) = \left\{
					\begin{array}{ll}
						\frac{1}{|I(y)|} & \text{if } I_{\ell}(y) = \emptyset \text{ or } I_s(y) = \emptyset  \\
						\frac{1}{2|I_{\ell}(y)|} & \text{if } I_{\ell}(y), I_s(y) \neq \emptyset \text{ and } i \in I_{\ell}(y) \\
						\frac{1}{2|I_{s}(y)|} & \text{if } I_{\ell}(y), I_s(y) \neq \emptyset \text{ and } i \in I_{s}(y).
					\end{array}
				\right.
			\end{equation*}
			Observe that $T[V_i \cup Y_i]$ is $k$-connected for every $i \in [t]$.

			We show that, with positive probability, for every $z \in Z$ there exists $i \in [t]$ such that $z$ has at least $k$ out- and in-neighbours in $V_i \cup Y_i$. If true, there is a partition $\{Z_1, \ldots, Z_{t}\}$ of $Z$ such that every $z \in Z_i$ has at least $k$ out- and in-neighbours in $V_i \cup Y_i$, for $i \in [t]$. Then $T[V_i \cup Z_i \cup Y_i]$ is $k$-connected, and so we can take $V_i' = V_i \cup Z_i \cup Y_i$ for $i \in [t]$.

			Fix $z \in Z$. We will show that $z$ has at least $k$ out- and in-neighbours in some set $V_i \cup Y_i$ with $i \in [t]$, with probability larger than $1 - 1/|Z|$. This, combined with a union bound, would prove \Cref{prop:extending-partition}, as explained in the previous paragraph.

			\def \Ip {I^+}
			\def \Im {I^-}
			Let $\Ip$ be the sets of indices $i \in [t]$ such that $z$ has at least $k$ out-neighbours in $V_i$, and define $\Im$ analogously for in-neighbours.
			Note that $z$ has either at least $k$ out-neighbours or at least $k$ in-neighbours in $V_i$, for each $i \in [t]$ (as $|V_i| \ge 2k$), implying $\Ip \cup \Im = [t]$.
			Thus, without loss of generality, $|\Ip| \ge t/2$. Note that if $\Ip \cap \Im \neq \emptyset$ then $z$ satisfies the above event with probability $1$. So we may assume $|\Im| \le t/2$, which implies, by assumption, that $z$ has at least $\max\{10^{10}|Z|, 100kt\}$ in-neighbours in $Y \cup V_1 \cup \ldots \cup V_t$.

			Let $\Np$ and $\Nm$ be the out- and in-neighbourhoods, respectively, of $z$ in $Y$. 
			We consider two cases, according to the size of $\Nm$.

			Suppose first that $|\Nm| \ge \min\{\frac{1}{2}\max\{10^{10} |Z|, 100kt\}, n/10^4\}$. Given a vertex $u \in \Nm$, it is in $\bigcup_{i \in \Ip} Y_i$ with probability at least $1/8$ (because $|I(u) \cap \Ip| \ge t/4$ and, for each $i \in I(y)$, the vertex $u$ is in $Y_i$ with probability at least $\frac{1}{2|I(y)|}$). It follows by Chernoff that $\left|\Nm \cap \bigcup_{i \in \Ip} Y_i\right| \ge |\Nm|/16 \ge kt$ (using $n \ge \tau_1 kt$), with probability at least $1 - \exp(-|\Nm|/64) \ge 1 - \exp(-\min\{n/10^{6}, |Z|\}) > 1 - 1 / |Z|$. Hence $z$ has at least $k$ in-neighbours in some set $Y_i$ (and at least $k$ out-neighbours in $V_i$) with $i \in \Ip$, with probability larger than $1 - 1 / |Z|$, as required.

			We now assume that $|\Nm| \le \min\{\frac{1}{2}\max\{10^{10}|Z|, 100kt\}, n/10^4\}$. Then $z$ has at least $\frac{1}{2}\max\{10^{10}|Z|, 100kt\} - |Z|$ in-neighbours in $\bigcup_i V_i$ and $|\Np| \ge |Y| - |\Nm| \ge |Y| - n/10^4$.
			For $I \subseteq \Im$, let $N'(I)$ be the set of vertices $u$ in $\Np$ that have at least $k$ out- and in-neighbours in $V_j$, for at least $|I| / 10^4$ values $j \in I$.
			We claim that $|N'(I)| \ge n / 10^4$. To see this, form an auxiliary bipartite graph $H$ with parts $\Np$ and $I$, such that $uj$ (with $u \in \Np$ and $j \in I$) is an edge of $H$ if $u$ has at least $k$ out- and in-neighbours in $V_j$. Then, by assumption on $Y$ and $\Np$, we have $e(H) \ge |I| \cdot (n / 10^3 - n / 10^4)$. Since $e(H) \le |N'(I)| \cdot |I| + n \cdot |I| / 10^4$, we find that $|N'(I)| \ge n / 10^4$, as claimed.

			Now, if $|\Im| > \frac{10^{10}t \log |Z|}{n}$, since every vertex in $N'(\Im)$ is in $\bigcup_{j \in \Im} Y_j$ with probability at least $|\Im| / 10^5 t$, the following holds
			\begin{equation*}
				\Big|N'(\Im) \cap \bigcup_{j \in \Im} Y_j\Big| 
				\ge \frac{|N'(\Im)| \cdot |\Im|}{10^6 t}
				\ge \frac{n |\Im|}{10^{10} \, t}
				\ge k|\Im|,
			\end{equation*}
			(using $n \ge \tau_1 kt$) with probability at least 
			\begin{equation*}
				1 - \exp\left(-\frac{n|\Im|}{10^{10} t}\right) > 1 - \exp(-\log |Z|) = 1 - 1/|Z|.
			\end{equation*}
			In particular, $|N'(\Im) \cap Y_j| \ge k$ for some $j \in \Im$, with probability larger than $1 - 1/|Z|$. Since $|N'(\Im) \cap Y_j| \ge k$ and $j \in \Im$ imply $z$ has at least $k$ out- and in-neighbours in $Y_j \cup V_j$, the requirements are satisfied in this case.

			Suppose thus that $|\Im| \le \frac{10^{10} t \log |Z|}{n}$.
			Then, there is $i_0 \in \Im$ such that $V_{i_0}$ has at least the following number of in-neighbours of $z$.
			\begin{equation*}
				\frac{\frac{1}{2}\max\{10^{10}|Z|,100kt\} - |Z| - kt}{|\Im|}
				\ge \frac{\frac{1}{4} \cdot 10^{10}|Z| + \frac{1}{4} \cdot 100kt - |Z| - kt}{\frac{10^{10} t \log |Z|}{n}}
				\ge \frac{24 ktn}{10^{10}t \log |Z|}
				\ge \frac{nk}{10^{10} \log n}.
			\end{equation*}
			In particular, $|V_{i_0}| \ge \frac{nk}{10^{10}\log n}$, and thus every $u \in N'(\{i_0\})$ satisfies $i_0 \in I_{\ell}(u)$, implying that $u \in Y_{i_0}$ with probability $\frac{1}{2|I_{\ell}(u)|} \ge \frac{k}{10^{11}\log n}$. Hence, the following holds
			\begin{equation*}
				|N'(\{i_0\}) \cap Y_{i_0}| 
				\ge \frac{nk}{10^{16} \log n} \ge k
			\end{equation*}
			with probability at least
			\begin{equation*}
				1 - \exp\left(-\frac{n}{10^{16} \log n}\right)
				\ge 1 - \exp(-\log n)
				\ge 1 - \frac{1}{n} 
				\ge 1 - \frac{1}{|Z|}.
			\end{equation*}
			As before, this shows that $z$ has at least $k$ out- and in-neighbours in $Y_{i_0} \cup V_{i_0}$, with probability at least $1 - 1/|Z|$, as required.
		\end{proof}

		Consider sets $V_1, \ldots, V_{t}$ as in \Cref{prop:connected-parts}. Let $Z := V \setminus (V_1 \cup \ldots \cup V_{t})$, and write $Y := Z \cap \Vg$, $Z_1 := Z \cap (\Vgp \setminus \Vgm)$, $Z_2 := Z \cap (\Vgm \setminus \Vgp)$, $Z_3 := Z \cap \Vb$ and $Z_4 := Z \setminus \Vo$.
		Let $\{Y_1,Y_2, Y_3, Y_4\}$ be a random partition of $Y$. We claim that the following properties hold simultaneously with positive probability.
		\begin{enumerate}[label = \rm(\alph*)]
			\item \label{itm:end-eligible}
				Every vertex in $Y$ has at least $k$ out- and in-neighbours in $V_i$, for at least $3t/4$ values $i \in [t]$.
			\item \label{itm:end-reservoir}
				At least $n/10^3$ vertices in $Y_i$ have at least $k$ out- and in-neighbours in $V_j$, for $i \in [4]$ and $j \in [t]$.
			\item \label{itm:end-degree-right}
				Every vertex in $Z_i$ either has at least $k$ out-neighbours in at least $3t/4$ sets $V_i$, or has at least $\max\{10^{10} |Z_i|, 100 kt\}$ out-neighbours in $Z_1 \cup \ldots \cup Z_{i-1} \cup Y_1 \cup \ldots \cup Y_i \cup V_1 \cup \ldots \cup V_{t}$; similarly for in-neighbours.
		\end{enumerate}
		Note that \ref{itm:end-eligible} holds deterministically. Indeed, given $y \in Y = \Vg$, it has out-neighbours in all but at most $kt$ sets $U(\alpha)$. Since each set $V_i$ contains at least $10k$ sets $U(\alpha)$ (this follows from the assumption that each $V_i$ contains a set $\B_j$, which in turn contains at least $10k$ sets $U(\alpha)$), it follows that the number of indices $i \in [t]$ for which $y$ has fewer than $k$ out-neighbours in $V_i$ is at most $kt / 9k = t/9$. An analogous argument holds for in-neighbours of $y$, showing that $y$ has at least $k$ out- and in-neighbours is at least $7t/9 \ge 3t/4$ sets $V_i$.

		Recall that, for every $j \in [t]$, there are at least $n / 100$ vertices in $Y$ that have at least $k$ out- and in-neighbours in $V_j$. Thus, by Chernoff and a union bound, \ref{itm:end-reservoir} holds with probability at least $1 - 4t \exp(-n/10^4) \ge 3/4$.

		We now show that \ref{itm:end-degree-right} holds with probability at least $15/16$ for $i = 1$. Similar arguments can be used to show that the same holds for every $i \in [4]$, proving that \ref{itm:end-degree-right} holds with probability at least $3/4$.

		Fix $z \in Z_1$. Since $z \in \Vgp$, it has at least $k$ out-neighbours in at least say $3t/4$ sets $V_i$, similarly to a paragraph above.
		Because $z$ is in $\Vo$ and by \ref{itm:gadget-bad-in}, $|\Nm(z) \cap \Vg| \ge \max\{10^{11}|\Vbm|, \tau_1 kt/2\}$. 
		Denote $V' := V_1 \cup \ldots \cup V_t$. Then $\Vg \subseteq Y \cup V'$. Thus, as $Z_1 \subseteq \Vbm$,
		\begin{equation*}
			\big|\Nm(z) \cap (Y \cup V')\big| 
			\ge \max\{10^{11}|\Vbm|, \tau_1 kt/2\}
			\ge \max\{10^{11} |Z_1|, 1000kt\}.
		\end{equation*}
		It follows, using Chernoff, that $\big|\Nm(z) \cap (Y_1 \cup V')\big| \ge \max\{10^{10}|Z_1|, 100kt\}$ with probability at least $1 - \exp(-\max\{10^{10}|Z_1|, 100kt\})$. Thus, \ref{itm:end-degree-right} holds with probability at least $1 - |Z_1|\exp(-\max\{10^{10}|Z_1|, 100kt\}) \ge 15/16$, as claimed.

		Apply \Cref{prop:extending-partition} four times: first with $Z_1, Y_1, V_1, \ldots, V_{t}$, then with $Z_2, Y_2, V_1', \ldots, V_{t}'$, where $V_1', \ldots, V_{t}'$ are the sets resulting from the first application of the claim, and so on. We end up with a partition $\{U_1, \ldots, U_{t}\}$ of $V$ where $T[U_i]$ is $k$-connected for $i \in [t]$, as required.
	\end{proof}

\section{Conclusion} \label{sec:conclusion}

	Recall that $f_t(k_1, \ldots, k_r)$ is the minimum $K$ such that the vertices of every strongly $K$-connected tournament can be partitioned into $t$ sets, the $i^{\text{th}}$ of which induces a strongly $k_i$-connected tournament. We showed that $f_t(k, \ldots, k) = O(kt)$, which is tight up to the implicit constant factor. It would be interesting to evaluate $f_t(k_1, \ldots, k_t)$ when possibly $k_1, \ldots, k_t$ vary significantly. 

	\begin{question}
		Is it true that $f_t(k_1, \ldots, k_t) = O(k_1 + \ldots + k_t)$?
	\end{question}
    Note that it would be enough to show that for every  $k_1\geq k_2$ one has that $f_2(k_1,k_2)=k_1+O(k_2)$.
    
	It would also be very interesting, but probably very hard, to determine if the analogue of $f_t(k_1, \ldots, k_t)$ for digraphs (which are not necessarily tournaments) holds. 

	\begin{question} [Question 1.3 in \cite{kuhn2016proof}]
		Is there a function $g$ such that, for every positive integer $k$, the vertices of every strongly $g(k)$-connected digraph can be partitioned into two sets inducing strongly $k$-connected subdigraphs?
	\end{question}

	Finally, we remark that Kim, K\"uhn and Osthus \cite{kim2016bipartitions} proved that for every integer $k \ge 1$ there exists $K$ such that if $T$ is a strongly $K$-connected tournament, then there is a partition $\{V_1, V_2\}$ of $V(T)$ such that $T[V_1]$, $T[V_2]$ and $T[V_1, V_2]$ are $k$-strongly connected. Denote the minimum such $K$ by $h(k)$. Their proof shows $h(k) = O(k^6 \log k)$. It would be interesting to determine the correct order of magnitude of $h(k)$.
	
	\begin{question}
		Is $h(k) = O(k)$?
	\end{question}

\subsection*{Acknowledgements}

	We would like to thank Lantao Zou for pointing out an error in a previous version of this paper. We would also like to thank the referees for helpful comments.

	\bibliography{main}
	\bibliographystyle{amsplain}
\end{document}